\documentclass[a4paper,twoside,11pt]{article}
\usepackage{mathrsfs}
\usepackage{amsmath,amsthm}
\usepackage{amsfonts,amssymb}
\usepackage{indentfirst}
\usepackage{cite}
\usepackage[a4paper,left=33mm,textwidth=14cm,textheight=21cm,headsep=4mm,top=25mm,bottom=25mm]{geometry}

\def\abs#1{|#1|}

\def\norm#1{\|#1\|}

\newcommand{\N}{{\mathbb N}}

\newcommand{\R}{{\mathbb R}}

\newcommand{\Z}{{\mathbb Z}}
\newcommand{\ft}{{\mathcal{F}}}

\renewcommand{\leq}{\leqslant}
\renewcommand{\geq}{\geqslant}

\newcommand{\M}{{\mathcal{M}}}

\newcommand{\supp}{{\mbox{supp\,}}}

\newcommand{\eps}{\varepsilon}
\newcommand{\Real}{\mathbb R}
\def\FF{{\mathscr{F}^{-1}}}
\def\F{{\mathscr{F}}}
\numberwithin{equation}{section}

\newtheorem{theorem}{Theorem}[section]

\newtheorem{proposition}[theorem]{Proposition}

\newtheorem{lemma}[theorem]{Lemma}
\newtheorem{corollary}[theorem]{Corollary}

\newtheorem{remark}[theorem]{Remark}

\begin{document}

\pagestyle{myheadings} \markboth{H. Zhang }{GWP for derivative
4NLS}

\title{Global well-posedness and scattering for the fourth order nonlinear Schr\"{o}dinger equations with small data}

\author{\bf Hua Zhang\\
{\small \it School of Mathematical Sciences, Peking University,
Beijing 100871, P. R. China}\\
{\small E-mail: zhanghuamaths@163.com(H. Zhang)}}
\date{}
 \maketitle
\thispagestyle{empty}
{\bf Abstract:} For $n\geq 3$, we study the Cauchy problem for the
fourth order nonlinear Schr\"{o}dinger equations, for which the
existence of the scattering operators and the global well-posedness
of solutions with small data in Besov spaces $B^{s}_{2,1}(\R^n)$ are
obtained. In one spatial dimension, we get the global well-posedness
result with small data in the critical homogeneous Besov spaces
$\dot{B}^{s}_{2,1}$. As a by-product, the existence of the
scattering  operators with small data is also obtained. In order to
show these results, the global version of the estimates for the
maximal functions and the local smoothing effects on the fourth
order Schr\"{o}dinger semi-groups are established.



\section{Introduction}
In the present paper, we consider the Cauchy problem for the fourth
order nonlinear Schr\"{o}dinger equations with derivatives (4NLS)
\begin{equation}\label{4NLS}
iu_{t}+\Delta^{2}u-\eps\Delta u=F((\partial_x^\alpha
u)_{\abs{\alpha}\leq 3},
(\partial_x^\alpha\bar{u})_{\abs{\alpha}\leq 3}),\quad
u(0,x)=u_0(x),
\end{equation}
where $\eps\in\{0,\, 1\}$,  $u$ is a complex valued function of
$(t,x)\in \R\times \R^n$,
\begin{equation*}
\Delta u=-\FF\abs{\xi}^{2}\F u,\quad \Delta^2 u=\FF\abs{\xi}^{4}\F
u,
\end{equation*}
$F:\mathbb{C}^{\frac{1}{3}n^3+2n^{2}+\frac{11}{3}n+2}\longrightarrow
\mathbb{C}$ is a polynomial of the form
\begin{equation}\label{Fd}
F(z)=P(z_{1},...,
z_{\frac{1}{3}n^3+2n^{2}+\frac{11}{3}n+2})=\sum_{m+1\leq\mid\beta\mid\leq
M+1}c_{\beta}z^{\beta},\quad c_{\beta}\in \mathbb{C}
\end{equation}
$m,\,M\in\N$ will be given below.

The fourth order nonlinear Schr\"odinger equation, including its
special forms, arise in  deep water wave dynamics, plasma physics,
optical communications (see \cite{Dys}). A large amount of
 work has been devoted to the  Cauchy problem of
dispersive equations, such as
\cite{Ben,Ca,Chihara,CS,GT,Guo1,Hao,Hayashi,Huo,Karpman1996,Kato,KP,KPV1,KPV2,KPV3,KPV98,KPV4,KF84,Pecher,Seg2003,Seg2004,Sj87,Ve88}
and references  therein. In \cite{Seg2003}, by using the method of
Fourier restriction norm, Segata studied a special fourth order
nonlinear Schr\"odinger equation in one dimensional space. And the
results have been improved in \cite{Huo,Seg2004}.

In order to study the influence of higher order dispersion on
solitary waves, instability and the collapse phenomena, Karpman
 introduced a class of nonlinear Schr\"odinger
equations (see \cite{Karpman1996})
\begin{equation*}
    i\Psi_t+\frac{1}{2}\Delta\Psi+\frac{\gamma}{2}\Delta^2\Psi+f(\abs{\Psi}^2)\Psi=0.
\end{equation*}

In \cite{Ben},  Ben-Artzi,  Koch and Saut discussed the sharp
space-time decay properties of fundamental solutions to the linear
equation
\begin{align*}
    i\Psi_t-\eps\Delta\Psi+\Delta^2\Psi=0,\
    \eps\in\{-1,0,1\}.
\end{align*}

In \cite{Guo1},  Guo and Wang considered the existence and
scattering theory for the Cauchy problem of nonlinear Schr\"odinger
equations with the form
\begin{equation}\label{w1}
    iu_t+(-\Delta)^m u+f(u)=0,\quad u(0,x)=\varphi(x).
\end{equation}
where $m\geq 1$ is an integer.  Pecher and Wahl in \cite{Pecher}
proved the existence of classical global solutions of \eqref{w1}
for the space dimensions $n\leq7m$ for the case $m\geq 1$. In
\cite{Hao}, Hao, Hsiao and Wang discussed the local well-posedness
of the Cauchy problem \eqref{w1} for the large initial data for
$m=2$ and $f(u)=P((\partial_x^j u)_{j\leq 2}, (\partial_x^j
\bar{u})_{j\leq 2})$ in one dimension.

  In \cite{Hao1}, Hao, Hsiao and Wang
considered the equation
\begin{equation}\label{w3}
i\partial_t u=\Delta^2-\eps\Delta u+P\left((\partial_x^\alpha
u)_{\abs{\alpha}\leq 2}, (\partial_x^\alpha
    \bar{u})_{\abs{\alpha}\leq2}\right),\quad
u(0,x)=\varphi(x).
\end{equation}
in multi-dimensional cases with $\eps=-1,0,1$,  where they
obtained the local well-posedness for the Cauchy problem
\eqref{w3} for the large initial data.

In \cite{KPV3}, Kenig,  Ponce and Vega  studied the nonlinear
Schr\"odinger equation of the form
$$
    \partial_{t} u=i\Delta u+P(u,\bar{u}, \nabla_x u,\nabla_x\bar{u}),\quad t\in\R,\ x\in\R^n
$$ and proved that the corresponding Cauchy problem is locally
well-posed for small data in the Sobolev spaces $H^s(\R^n)$ and in
its weighted version by pushing forward the linear estimates
associated with the Schr\"odinger group
$\{e^{it\Delta}\}_{-\infty}^\infty$ and by introducing suitable
function spaces where these estimates act naturally. They also
studied generalized nonlinear Schr\"odinger equations in
\cite{KPV98} and quasi-linear Schr\"odinger equations in
\cite{KPV4}. In one dimensional case, the smallness assumption on
the size of the data was removed by  Hayashi and  Ozawa
\cite{Hayashi} by using a change of variables to obtain an
equivalent system with a nonlinear term independent of $\partial_x
u$, where the new system could be treated by the standard energy
method. In \cite{Chihara},  Chihara was able to remove the size
restriction on the data in any dimensions by using an invertible
classical pseudo-differential operator of order zero.

In \cite{Wang}, Wang and Wang  discussed
\begin{equation}
iu_{t}+\Delta_{\pm}u=F(u,\bar{u},\nabla u,\nabla
 \bar{u}),\quad u(0,x)=u_{0}(x)\label{w5}
 \end{equation}
where $\Delta_{\pm}u=\sum_{i=1}^n \eps_i\partial_{x_i}^2u$ and
$\eps_i\in\{-1,1\}$.  They established an estimate for the global
maximal function and proved \eqref{w5} is global well-posed.
Moreover, t+he existence of the scattering operators to \eqref{w5}
with small data in Besov space $B^{s}_{2,1}(\R^{n})$ are obtained.

 In this paper, we mainly use the dispersive
smoothing effects of the linear Schr\"{o}dinger equation(cf. Kenig,
Ponce and Vega \cite{KPV3,KPV98}). The crucial point is that the
fourth order Schr\"{o}dinger semi-groups has the following local
smoothing effects for $n \geq2$:
\begin{equation}
\sup_{\alpha}\norm{e^{it(\Delta^2-\eps\Delta)}u_0}_{L^{2}_{t,x}(\R\times
Q_{\alpha})}\lesssim
\norm{u_0}_{\dot{H}^{-\frac{3}{2}}}\label{Kenig1}
\end{equation}
\begin{equation}
\sup_{\alpha}\norm{D^{3}\int_{0}^{t}
e^{i(t-\tau)(\Delta^2-\eps\Delta)}f(s)ds}_{L^{2}_{t,x}(\R\times
Q_{\alpha})}\lesssim \sum_{\alpha}\norm{f}_{L^{2}_{t,x}(\R\times
Q_{\alpha})}\label{Bh}
\end{equation}
where
$D^{3}u=\sum\limits_{\abs{\beta}=3}\norm{\partial^{\beta}_{x}u}$,
$Q_{\alpha}$ is the unit cube with center at $\alpha\in\Z^n$. The
estimate \eqref{Kenig1} was established by Kenig, Ponce and  Vega in
\cite{KPV3}. Since the estimate \eqref{Bh} contains three order
smoothing effect, with the method of Wang and Wang \cite{Wang}, we
can control three order derivative nonlinear terms in \eqref{4NLS}.
The estimate \eqref{Bh} improves the local smoothing effect
estimates of Hao, Hsiao and Wang \cite{Hao1}, see Proposition
\ref{p36} and Lemma \ref{l41} for details.

\subsection{Main results}

In this paper, we mainly use the method in \cite{Wang}
 to study the global well-posedness and the existence of the
scattering operator of \eqref{4NLS} with small data in
$B^{s}_{2,1},s>n/2+9/2$. We now state our main results, the
notations used in this paper can be found at the end of this part.
\begin{theorem}\label{t11}
Let $n\geq 3$ and $s>n/2+9/2$. Let $F(z)$ be as in \eqref{Fd} with
$2+8/n\leq m\leq M<\infty$. We have the following results.

(1) If $\norm{u_0}_{B^{s}_{2,1}}\leq \delta$ for $n\geq5$, or
$\norm{u_0}_{B^{s}_{2,1}\cap \dot{H}^{-\frac{3}{2}}}\leq \delta$ for
  $n=3,4$, where $\delta>0$ is a suitably small number, then \eqref{4NLS} has
  a unique global solution $u\in C(\R,B^{s}_{2,1})\cap
  X_0$, where
\begin{align*}
X_{0}=\left\{u:
\begin{array}{l}
\norm{D^{\beta}u}_{\ell^{1,s-3/2}_{\triangle}\ell^{\infty}_{\alpha}(L^{2}_{t,x}(\R\times
Q_{\alpha}))}\lesssim\delta,\; |\beta|\leq3\\
\norm{D^{\beta}u}_{\ell^{1,s-3/2}_{\triangle}\ell^{2+8/n}_{\alpha}(L^{\infty}_{t,x}\cap
(L^{2m}_{t}L^{\infty}_{x})(\R\times
Q_{\alpha}))}\lesssim\delta,\;|\beta|\leq3
\end{array}
\right\}.
\end{align*}
 Moreover, for $ n\geq 5$, the scattering operator of \eqref{4NLS}
carries the ball $\{u:\norm{u}_{B^{s}_{2,1}}\leq\delta\}$ into
$B^{s}_{2,1}$.

 (2) If $s+9/2\in\N$ and
$\norm{u_0}_{H^s}\leq\delta$ for
  $n\geq5$, or $\norm{u_0}_{{H^s}\cap{\dot{H}^{-3/2}}}\leq\delta$ for $n=3,4$, where $\delta>0$ is a suitably small number, then  \eqref{4NLS} has
  a unique global solution $u\in C(\R,H^{s})\cap
  X$, where
\begin{align*}
X=\left\{u:
\begin{array}{l}
\norm{D^{\beta}u}_{\ell^{\infty}_{\alpha}(L^{2}_{t,x}(\R\times
Q_{\alpha}))}\lesssim\delta,\;\abs{\beta}\leq s+3/2\\
\norm{D^{\beta}u}_{\ell^{2+8/n}_{\alpha}(L^{\infty}_{t,x}\cap
(L^{2m}_{t}L^{\infty}_{x})(\R\times
Q_{\alpha}))}\lesssim\delta,\;|\beta|\leq 3
\end{array}
\right\}.
\end{align*}
Moreover, for $n\geq5$, the scattering operator of \eqref{4NLS}
carries the ball $\{u:\norm{u}_{H^{s}}\leq\delta\}$ into $H^{s}$.
\end{theorem}

Next, we consider one spatial dimension case. Denote
\begin{equation*}
s_{k}=\frac{1}{2}-\frac{4}{k},\quad
\tilde{s}_{k}=\frac{1}{2}-\frac{1}{k}.
\end{equation*}

\begin{theorem}\label{t14}
Let  $n=1$, $\eps=0$, $M\geq m\geq 8$, and $u_{0}\in
\dot{B}^{3+\tilde{s}_{M}}_{2,1}\cap\dot{B}^{s_{m}}_{2,1}$.  Assume
that there exists a small $\delta>0$ such that
$\norm{u_0}_{\dot{B}^{1+\tilde{s}_{M}}_{2,1}\cap\dot{B}^{1+s_{m}}_{2,1}}\leq\delta$.
Then \eqref{4NLS} has a global solution $u\in
X=\{u\in\mathscr{S}'(\R^{1+1}):\norm{u}_{X}\lesssim \delta\}$, where
\begin{align*}
\norm{u}_{X}=&\sup_{s_{m}\leq s\leq
\tilde{s}_{M}}\sum_{i=0,3}\sum_{j\in\mathbb{Z}}
|\!|\!|\partial^{i}_{x}\triangle_{j}u|\!|\!|_{s}
\ \ \text{for}\ \  m>8,\nonumber\\
\norm{u}_{X}=&\sum_{i=0,3}\big(\norm{\partial^{i}_{x}u}_{L^{\infty}_{t}L^{2}_{x}\cap
L^{10}_{t,x}}+\sup_{s_{9}\leq s\leq
\tilde{s}_{M}}\sum_{j\in\mathbb{Z}}\||\partial^{i}_{x}\triangle_{j}u\||_{s}\big)
\
\ \text{for}\ \  m=8,\nonumber\\
|\!|\!|\triangle_{j}v|\!|\!|_{s}:=&2^{sj}\big(\norm{\triangle_{j}v}_{L^{\infty}_{t}L^{2}_{x}\cap
L^{10}_{t,x}}+2^{\frac{3j}{2}}\norm{\triangle_{j}v}_{L^{\infty}_{x}L^{2}_{t}}\big)\nonumber\\
&+2^{(s-\tilde{s}_{m})j}\norm{\triangle_{j}v}_{L^{m}_{x}L^{\infty}_{t}}+2^{(s-\tilde{s}_{M})j}\norm{\triangle_{j}v}_{L^{M}_{x}L^{\infty}_{t}}.
\end{align*}
\end{theorem}
Recall that the norm on homogeneous Besov spaces $\dot{B}^{s}_{2,1}$
can be defined in the following way:
$$\norm{f}_{\dot{B}^{s}_{2,1}}=\sum^{\infty}_{j=-\infty}2^{sj}\left(\int_{2^{j}\leq\abs{\xi}<2^{j+1}}|\mathscr{F} f(\xi)|^{2}d\xi\right)^{1/2}.$$

\subsection{Notations}
Throughout this paper,  we will always use the following notations.
$\mathscr{S}(\R^n)$ and $\mathscr{S}'(\R^n)$ stand for the Schwartz
space and its dual space, respectively.  We denote by $L^p(\R^n)$
the Lebesgue spaces with norms $\|\cdot\|_p:=
\|\cdot\|_{L^p(\R^n)}$. The Bessel potential space is defined by
$H^s_p(\R^n): =(I-\Delta)^{-s/2} L^p(\R^n)$, $H^s(\R^n)=H^s_2
(\R^n)$, $\dot H^s(\R^n) =(-\Delta)^{-s/2}
L^2(\R^n)$.\footnote{$\R^n$ will be omitted in the definitions of
various function spaces if there is no confusion.} For any
quasi-Banach space $X$, we denote by $X^*$ its dual space,
 by $L^p(I, X)$ the Lebesgue-Bochner space,  $\|f\|_{L^p(I,X)}:=
(\int_I \|f(t)\|^p_X dt)^{1/p}$.  If $X=L^r(\Omega)$, then we write
$L^p(I, L^r(\Omega))= L^p_tL^r_x(I\times \Omega)$ and
$L^p_{t,x}(I\times \Omega)= L^p_tL^p_x(I\times \Omega)$.   Let
$Q_\alpha$ be the unit cube with center at $\alpha\in \mathbb{Z}^n$,
i.e., $Q_\alpha=\alpha+ Q_0, Q_0= \{x=(x_1,...x_n): -1/2\le x_i<
1/2\}.$  We also need the function space $\ell^q_\alpha
(L^p_tL^r_{x}(I\times Q_\alpha))$ with the norm
$$
\|f\|_{\ell^q_\alpha (L^p_{t}L^r_{x}(I\times Q_\alpha))}:=
\left(\sum_{\alpha\in \mathbb{Z}^n} \|f\|^q_{L^p_{t}L^r_{x}(I\times
Q_\alpha)} \right)^{1/q}.
$$
We denote by $\mathscr{F}$ ($\mathscr{F}^{-1}$) the (inverse)
Fourier transform for the spatial variables; by $\mathscr{F}_t$
($\mathscr{F}^{-1}_{t}$) the (inverse) Fourier transform for the
time variable and by $\mathscr{F}_{t,x}$ ($\mathscr{F}^{-1}_{t,x}$)
the (inverse) Fourier transform for both time and spatial variables,
respectively.  If there is no additional explanation, we always
denote by $\varphi_k(\cdot)$ the dyadic decomposition functions as
in \eqref{dyadic-funct}; and by $\sigma_k(\cdot)$ the uniform
decomposition functions as in \eqref{Mod.2}.  $u \star v$ and $u*v$
will stand for the convolution on time and on spatial variables,
respectively, i.e.,
$$
(u\star v) (t,x)= \int_{\R} u(t-\tau,x) v(\tau,x)d\tau, \ \ (u* v)
(t,x)= \int_{\R^n} u(t,x-y) v(t,y)dy.
$$
Symbols $\R, \mathbb{N}$ and $ \mathbb{Z}$ will stand for the sets
of real numbers, natural numbers and integers, respectively.  $c<1$,
$C>1$ will denote positive universal constants, which may be
different at different places.  $a\lesssim b$ stands for $a\le C b$
for some constant $C>1$, $a\sim b$ means that $a\lesssim b$ and
$b\lesssim a$.  We denote by $p'$ the dual number of $p \in
[1,\infty]$, i.e., $1/p+1/p'=1$. For any $a>0$, we denote by $[a]$
the minimal integer that is larger than or equals to $a$. $B(x,R)$
will denote the ball in $\R^n$ with center at $x$ and radial $R$. We
denote $S_{\eps}(t)=e^{it(\Delta^{2}-\eps\Delta)}$ and
$\mathscr{A}_{\eps}f=\int_{0}^{t}S_{\eps}(t-\tau)f(\tau)d\tau$.
\subsection{Besov  spaces} \label{functionspace}
Let us recall that Besov spaces $B^s_{p,q}:=B^s_{p,q}(\R^n)$ are
defined as follows (cf. \cite{bergh,Triebel}).  Let $\psi: \R^n \to
[0,1]$ be a smooth radial bump function adapted to the ball
$B(0,2)$:
\begin{align}
\psi (\xi)=\left\{
\begin{array}{ll}
1, & |\xi|\leq 1,\\
{\rm smooth}, & |\xi|\in (1,2),\\
 0, & |\xi|\geq 2.
\end{array}
\right.
 \label{cutoff}\
\end{align}
We write $\delta(\cdot):= \psi(\cdot)-\psi(2\,\cdot)$ and
\begin{align}
\varphi_j:= \delta(2^{-j}\cdot) \ \ {\rm for} \ \  j\geq 1; \quad
\varphi_0:=1- \sum_{j\geq 1} \varphi_j.  \label{dyadic-funct}
\end{align}
We say that $ \triangle_j := \mathscr{F}^{-1} \varphi_j \mathscr{F},
\quad j\in \mathbb{N} \cup \{0\}$ are the dyadic decomposition
operators.  Besov spaces $B^s_{p,q}=B^s_{p,q}(\R^n)$ are defined in
the following way:
\begin{align}
B^s_{p,q} =\left \{ f\in \mathscr{S}'(\R^n): \; \|f\|_{B^s_{p,q}} =
\left(\sum^\infty_{j=0}2^{sjq} \|\,\triangle_j f\|^q_p
\right)^{1/q}<\infty \right\}. \label{Besov.1}
\end{align}
 Let $\rho\in
\mathscr{S}(\R^n)$ and $\rho:\, \R^n\to [0,1]$ be a smooth radial
bump function adapted to the ball $B(0, \sqrt{n})$, say
$\rho(\xi)=1$ as $|\xi|\le \sqrt{n}/2$, and $\rho(\xi)=0$ as
$|\xi| \ge \sqrt{n} $. Let $\rho_k$ be a translation of $\rho$: $
\rho_k (\xi) = \rho (\xi- k), \; k\in \mathbb{Z}^n$.  We write
(see \cite{Wang,Wang1,Wang2,Wang3})
\begin{align}
\sigma_k (\xi)= \rho_k(\xi) \left(\sum_{k\in \mathbb{
Z}^n}\rho_k(\xi)\right)^{-1}, \quad k\in \mathbb{Z}^n. \label{Mod.2}
\end{align}

 We can define the space $\ell^{1,s}_\triangle \ell^q_\alpha (L^p_{t,x}(I\times
Q_\alpha))$ with the following norm:
\begin{align}
\|f\|_{\ell^{1,s}_\triangle \ell^q_\alpha (L^p_{t} L^r_x (I\times
Q_\alpha))}:= \sum^\infty_{j=0} 2^{sj} \left(\sum_{\alpha\in
\mathbb{Z}^n} \|\triangle_j f\|^q_{L^p_{t}L^r_x (I\times Q_\alpha)}
\right)^{1/q}. \label{Besov.7}
\end{align}
A special case is $s=0$,  $\ell^{1,0}_\triangle \ell^q_\alpha
(L^p_{t}L^r_x(I\times Q_\alpha)) = \ell^{1}_\triangle \ell^q_\alpha
(L^p_{t}L^r_x (I\times Q_\alpha))$.
 The rest of this paper is organized as follows. In Section 2, we
 give the details of the estimates for the maximal function in
 certain function spaces. Section 3 is devoted to consider the
 spatial local versions for the Strichartz estimates and giving
 some remarks on the estimates of the local smoothing effects. In
 Sections 4-5, we prove our main Theorems \ref{t11}-\ref{t14}, respectively.

 \section{Estimates for the maximal function}
 We give some estimates for the maximal function related to the
 fourth order Schr\"odinger semi-groups,  an earlier time-local
 version is due to Kenig, Ponce and vega \cite{KPV2}. We give some time-global versions by using a different approach.
  These estimates are crucial in the proof of Theorems \ref{t11}-\ref{t14}, respectively.

 \subsection{Time-local version}
 Recall that
 $S_{\eps}(t)=e^{it(\Delta^2-\eps\Delta)}=\ft^{-1}e^{it(\abs{\xi}^{4}+\varepsilon\abs{\xi}^{2})}\ft$, where
\begin{equation*}
 \abs{\xi}^{4}=\left(\sum_{j=1}^{n}{\xi_{j}^{2}}\right)^{2},\abs{\xi}^{2}=\sum_{j=1}^{n}\xi_{j}^{2}.
\end{equation*}
Hao, Hsiao and Wang \cite{Hao1} established the following maximal
function estimate
\begin{align*}
\left(\sum_{\alpha\in
\Z^n}\norm{S_{\varepsilon}(t)u_0}^{2}_{L^{\infty}_{t,x}([0,T]\times
Q_\alpha)}\right)^{1/2}\lesssim C(T)\norm{u_0}_{H^s}
\end{align*}
where $s>n+1/2$, $T\in(0,1]$.

\subsection{Time-global version}
 Recall that we have the following equivalent norm on Besov
 spaces \cite{bergh,Triebel}:
\begin{lemma}\label{l22}
Let $1\leq p,q\leq\infty$, $\sigma>0$, $\sigma\notin\N$. Then we
have
\begin{equation*}
\norm{f}_{B^{\sigma}_{p,q}}\sim\sum_{|\beta|\leq|\sigma|}\norm{D^{\beta}f}_{L^{p}(\R^n)}
+\sum_{|\beta|\leq|\sigma|}\left(\int_{\R^n}|h|^{-n-q\{\sigma\}}\norm{\Delta_{h}D^{\beta}f}^{q}_{L^{p}(\R^n)}dh\right)^{1/q}
\end{equation*}
where $\Delta_{h}f=f(\cdot+h)-f(\cdot)$, $[\sigma]$ denotes the
minimal integer that is larger than or equals to $\sigma$,
$\{\sigma\}=\sigma-[\sigma]$.
\end{lemma}
 Taking $p=q$ in Lemma
\ref{l22}, we get
\begin{equation*}
\norm{f}_{B^{\sigma}_{p,p}}\sim\sum_{|\beta|\leq|\sigma|}
\norm{D^{\beta}f}_{L^{p}(\R^n)}
+\sum_{|\beta|\leq|\sigma|}\left(\int_{\R^n}|h|^{-n-p\{\sigma\}}\norm{\Delta_{h}D^{\beta}f}^{p}_{L^{p}(\R^{n})}dh\right)^{1/p}.
\end{equation*}
\begin{lemma}\label{l23}
Let $1<p<\infty$, $s>1/p$. Then we have
\begin{equation*}
\left(\sum_{\alpha\in \R^n}\norm{u}^{p}_{L^{\infty}_{t,x}(\R\times
Q_\alpha)}\right)^{1/p}\lesssim\norm{(I-\partial^{2}_{t})^{s/2}u}_{L^{p}(\R,B^{ns}_{p,p}(\R^{n}))}.
\end{equation*}
\end{lemma}
The proof of lemma \ref{l23} can be found in \cite{Wang}.

For the case $\eps=1$, we recall some results of Guo and Wang
\cite{Guo}. For $S_{1}(t)=e^{it(\Delta^{2}-\Delta)}$,  we have
\begin{proposition}\label{pp}
 Assume $2\leq p\leq\infty$, $\frac{1}{p}+\frac{1}{p'}=1$, $1\leq q\leq \infty$, $\delta=\frac{1}{2}-\frac{1}{p}$, $-2n\delta\leq
 s'-s$, then
\begin{equation*}
\norm{S_{1}(t)g}_{B^{s}_{p,q}}\lesssim k(t)\norm{g}_{B^{s'}_{p',q}},
\end{equation*}
where
\begin{align*}
k(t)=\left\{
\begin{array}{ll}
\abs{t}^{\frac{1}{4}\min(s'-s-2n\delta,0)},& 0<t\leq1,\\
\abs{t}^{-n\delta},& t>1.
\end{array}
\right.
\end{align*}
\end{proposition}
In particular, if we choose $s'=s=0$ in the above proposition, then
we have
\begin{equation*}
\norm{S_{1}(t)g}_{B^{0}_{p,q}}\lesssim k(t)\norm{g}_{B^{0}_{p',q}},
\end{equation*}
where
\begin{align}\label{Guo4}
k(t)=\left\{
\begin{array}{ll}
\abs{t}^{-\frac{1}{2}n\delta},& 0<t\leq1,\\
\abs{t}^{-n\delta},& t>1.
\end{array}
\right. \
\end{align}
Combining \eqref{Guo4} with the Strichartz estimate in \cite{Guo},
we have an especial version proposition of Guo, Peng and Wang
\cite{Guo}.
\begin{proposition}\label{Guo1}
Let $A_{1}f=\int_{0}^{t}S_{1}(t-\tau)f(\tau)d\tau$, for $2\leq
p,q\leq\infty$,
$\frac{n}{2}(\frac{1}{2}-\frac{1}{p})<\frac{2}{q}<n(\frac{1}{2}-\frac{1}{p})$,
\begin{align*}
\norm{S_{1}(t)f}_{L^{q}_{t}L^{p}_{x}(\R\times
\R^n)}\lesssim &\norm{f}_{L^{2}},\\
\norm{A_{1}f}_{L^{q}_{t}L^{p}_{x}(\R\times \R^n)}\lesssim
&\norm{f}_{L^{q'}_{t}L^{p'}_{x}(\R\times \R^n)},\\
\norm{A_{1}f}_{L^{\infty}_{t}L^{2}_{x}(\R\times \R^n)}\lesssim
&\norm{f}_{L^{q'}_{t}L^{p'}_{x}(\R\times \R^n)},\\
\norm{A_{1}f}_{L^{q}_{t}L^{p}_{x}(\R\times \R^n)}\lesssim
&\norm{f}_{L^{1}_{t}L^{2}_{x}(\R\times \R^n)}.
\end{align*}
\end{proposition}
When $\eps=1$, the endpoint Strichartz estimate also hold. In fact,
we have
\begin{proposition}\label{eps}
Let $n\geq5$, $2\leq p$, $\rho\leq 2n/(n-4)$$(2\leq p$,
$\rho<\infty$, if $n=4)$, $4/\gamma(\cdot)=n(1/2-1/\cdot)$.  We have
\begin{align}
\norm{S_{1}(t)u_0}_{L^{\gamma(p)}(\R,L^{p}(\R^{n}))}\lesssim&\norm{u_0}_{L^{2}(\R^{n})},\label{str3}\\
\norm{\mathscr{A}_{1}F}_{L^{\gamma(p)}(\R,L^{p}(\R^{n}))}\lesssim&\norm{F}_{L^{\gamma(\rho)\prime}(\R,L^{\rho\prime}(\R^{n}))}.\label{str4}
\end{align}
\end{proposition}
\begin{proof}
Because the proof is similar to Keel and Tao \cite{Tao2}, we only
give the sketch. Let $2^{**}=2n/(n-4)$, we need to prove
\begin{equation}
\norm{S_{1}(t)u_{0}}_{L^{2}(\R,L^{2^{**}})}\lesssim\norm{u_{0}}_{L^2}
\end{equation}
let $u_{0}=\varphi$, for $\psi \in C^{\infty}$. According to dual
estimate, we need to prove
\begin{equation}\label{turn1}
\abs{\int_{-T}^{-T}(S_{1}(t)\varphi,\psi(t))dt}\lesssim\norm{\varphi}_{L^2}\norm{\psi}_{L^{2}(-T,T;L^{(2^{**})'})}
\end{equation}
By the $TT^*$ method  and symmetry, (\ref{turn1}) is in turn
equivalent to the bilinear form estimate
\begin{equation}\label{turn2}
\abs{L(F,G)}\lesssim\norm{F}_{L^{2}(-T,T;L^{(2^{**})'})}\norm{G}_{L^{2}(-T,T;L^{(2^{**})'})}
\end{equation}
Here $L(F,G)=\int\int_{D}(S_{1}(t-s)F(s),G(t))dsdt, D=\{(s,t)\in
[-T,T]^{2}, s\leq t\}$.

Now we  take a procedure which is suit our proof, one can see Wang
\cite{Wangbook} for details. We decompose $L(F,G)$ dyadically as
$\sum_{j}L_{j}(F,G)$. Here
\begin{equation}\label{turn3}
L_{j}(F,G)=\int\int_{D_{j}}(S_{1}(t-s)F(s),G(t))dsdt
\end{equation}
For simplicity, we assume $F,G \in C_{0}^{\infty}(I), I=[-T,T]$.
From Proposition \ref{pp},  for $\frac{1}{p}+\frac{1}{p'}=1, 2\leq
p\leq\infty$, we have
\begin{equation}\label{decays}
\norm{S_{1}(t)u_{0}}_{p}\lesssim(\abs{t}^{\frac{n}{2}}+\abs{t}^{\frac{n}{4}})^{\frac{2}{p}-1}\norm{u_{0}}_{p'}
\end{equation}
Using (\ref{decays}), similar to lemma 4.1 in \cite{Tao2},  we can
prove
\begin{equation}\label{turn4}
\abs{L_{j}(F,G)}\lesssim[(2^{j}T)^{\beta_{1}(a,b)}+(2^{j}T)^{\beta_{2}(a,b)}]^{-1}\norm{F}_{L^{2}_{I}(L^{a'})}\norm{G}_{L^{2}_{I}(L^{b'})}
\end{equation}
holds for all $j\in\Z$ and all $(\frac{1}{a},\frac{1}{b})$ in a
neighbourhood of $(\frac{1}{2^{**}},\frac{1}{2^{**}})$.

 Here
 $\beta_{1}(a,b)=\frac{2}{\gamma(b)}+\frac{2}{2\gamma(\frac{2a}{b})}-1$
 and
 $\beta_{1}(a,b)=\frac{4}{\gamma(b)}+\frac{4}{2\gamma(\frac{2a}{b})}-1$.
For later bilinear interpolation, we need
\begin{lemma}(\cite{bergh}, Section 3.13.5(b)) If $A_{0}, A_{1}, B_{0}, B_{1}, C_{0},
C_{1}$ are Banach spaces, and the bilinear operator $T$ is bounded
from
\begin{eqnarray*}
T:A_{0}\times B_{0}\rightarrow C_{0}\\
T:A_{0}\times B_{1}\rightarrow C_{1}\\
T:A_{1}\times B_{0}\rightarrow C_{1}
\end{eqnarray*}
then whenever $0<\theta_{0}, \theta_{1}<1, 1\leq p, q, r\leq\infty$
are such that $1\leq\frac{1}{p}+\frac{1}{q}$ and
$\theta=\theta_{0}+\theta_{1}$, one has
$$T:(A_{0},A_{1})_{\theta_{0}, pr}\times(B_{0},B_{1})_{\theta_{1},
qr}\rightarrow(C_{0},C_{1})_{\theta, r}$$.
\end{lemma}
To get the result we need, let
$A_{0}=B_{0}=L^{2}_{t}(L^{a_{0}'}_{x})$ and
$A_{1}=B_{1}=L^{2}_{t}(L^{a_{1}'}_{x})$. We should make the
following equations hold
\begin{align*} \left\{
\begin{array}{l}
\beta_{1}(a_{0}, a_{0})(1-\theta)+\beta_{1}(a_{0}, a_{1})\theta=0\\
\frac{1}{(2^{**})'}=\frac{1-\theta_{1}}{a_{0}'}+\frac{\theta_{2}}{a_{1}'}\\
\theta_{1}+\theta_{2}=\theta
\end{array}
\right.
\end{align*}
There are enough space to choose $\theta_{1}, \theta_{2}$. For
example, we can select $\theta_{1}=\theta_{2}=\frac{1}{3}$ to make
the above equations hold.

Moreover, we also need
\begin{align*} \left\{
\begin{array}{l}
\beta_{2}(a_{0}, a_{0})(1-\eta)+\beta_{2}(a_{0}, a_{1})\eta=0\\
\frac{1}{(2^{**})'}=\frac{1-\eta_{1}}{a_{0}'}+\frac{\eta_{2}}{a_{1}'}\\
\eta_{1}+\eta_{2}=\eta
\end{array}
\right.
\end{align*}
the situation is similar, we omit it.
\end{proof}

For the semi-group $S_{0}(t)$, we have the following Strichartz
estimate (cf.\cite{Tao2}):
\begin{proposition}\label{p24}
Let $n\geq4$, $2\leq p$, $\rho\leq 2n/(n-4)$$(2\leq p$,
$\rho<\infty$, if $n=4)$, $4/\gamma(\cdot)=n(1/2-1/\cdot)$.  We have
\begin{align}
\norm{S_{0}(t)u_0}_{L^{\gamma(p)}(\R,L^{p}(\R^{n}))}\lesssim&\norm{u_0}_{L^{2}(\R^{n})},\label{str3}\\
\norm{\mathscr{A}_{0}F}_{L^{\gamma(p)}(\R,L^{p}(\R^{n}))}\lesssim&\norm{F}_{L^{\gamma(\rho)\prime}(\R,L^{\rho\prime}(\R^{n}))}.\label{str4}
\end{align}
\end{proposition}

If both $p$ and $\rho$ are equal to $2n/(n-4)$, then \eqref{str3}
and \eqref{str4} are said to be the endpoint Strichartz estimates.
Using Proposition \ref{Guo1}-\ref{p24}, we have
\begin{proposition}\label{p25}
Let $2^{*}=2+8/n$. For any $p\geq2^{*},s>n/2$, we have
\begin{equation*}
\left(\sum_{\alpha\in\mathbb{Z}^{n}}\norm{S_{\eps}(t)u_{0}}^{p}_{L^{\infty}_{t,x}(\R\times
Q_{\alpha})}\right)^{1/p}\lesssim\norm{u_0}_{H^{s}}
\end{equation*}
\end{proposition}

\begin{proof}
For short, we write $\langle \partial_t
\rangle=(I-\partial^{2}_{t})^{1/2}$. By Lemma~\ref{l23}, for any
$s_{0}>1/2^{*}$,
\begin{align}
\left(\sum_{\alpha\in\mathbb{Z}^{n}}\norm{S_{\eps}(t)u_{0}}^{p}_{L^{\infty}_{t,x}(\R\times
Q_{\alpha})}\right)^{1/p}\lesssim
&\left(\sum_{\alpha\in\mathbb{Z}^{n}}\norm{S_{\eps}(t)u_{0}}^{2^{*}}_{L^{\infty}_{t,x}(\R\times
Q_{\alpha})}\right)^{1/2^{*}}\nonumber\\
\lesssim &\norm{\langle \partial_t
\rangle^{s_{0}}S_{\eps}(t)u_0}_{L^{2^{*}}(\R,B^{ns_{0}}_{2^{*},2^{*}}(\mathbb{R}^{n}))}.\label{mfp}
\end{align}
We have
\begin{equation*}
\norm{\langle \partial_t
\rangle^{s_{0}}S_{\eps}(t)u_0}_{L^{2^{*}}(\R,B^{ns_{0}}_{2^{*},2^{*}}(\mathbb{R}^{n}))}
=\sum_{k=0}^{\infty}2^{ns_{0}k2^{*}}\norm{\langle
\partial_t
\rangle^{s_{0}}\triangle_{k}S_{\eps}(t)u_0}_{L^{2^{*}}_{t,x}(\R^{1+n})}.
\end{equation*}
Using the dyadic decomposition to the time-frequency, we obtain that
\begin{equation*}
\norm{\langle \partial_t
\rangle^{s_{0}}\triangle_{k}S_{\eps}(t)u_0}_{L^{2^{*}}_{t,x}(\R^{1+n})}\lesssim\sum_{j=0}^{\infty}\norm{\mathscr{F}^{-1}_{t,x}\langle
\tau
\rangle^{s_{0}}\varphi_{j}(\tau)\mathscr{F}_{t}e^{it(\abs{\xi}^{4}+\eps\abs{\xi}^{2})}\varphi_{k}(\xi)
\mathscr{F}_{x}u_{0}}_{L^{2^{*}}_{t,x}(\mathbb{R}^{1+n})}.
\end{equation*}
For convenience, we let
$h_{\eps}(\xi)=\abs{\xi}^{4}+\eps\abs{\xi}^{2}$. Observing that
\begin{equation}
(\mathscr{F}^{-1}_{t}\langle \tau
\rangle^{s_{0}}\varphi_{j}(\tau))\star
e^{ith_{\eps}(\xi)}=ce^{ith_{\eps}(\xi)}\varphi_{j}(h_{\eps}(\xi))\langle
h_{\eps}(\xi)\rangle^{s_0},\label{convolution}
\end{equation}
and using the Strichartz's inequality and Plancherel's identity, we
get
\begin{align}
\norm{\langle \partial_t
\rangle^{s_{0}}\triangle_{k}S_{\eps}(t)u_0}_{L^{2^{*}}_{t,x}(\R^{1+n})}\lesssim
&\sum_{j=0}^{\infty}\norm{S_{\eps}(t)\F^{-1}_{x}\langle h_{\eps}(\xi)\rangle^{s_{0}}\varphi_{j}(h_{\eps}(\xi))\varphi_{k}(\xi)\mathscr{F}_{x}u_{0}}_{L^{2^{*}}_{t,x}(\mathbb{R}^{1+n})}\nonumber\\
\lesssim
&\sum_{j=0}^{\infty}\norm{\F^{-1}_{x}\langle h_{\eps}(\xi)\rangle^{s_{0}}\varphi_{j}(h_{\eps}(\xi))\varphi_{k}(\xi)\mathscr{F}_{x}u_{0}}_{L^{2}_{x}(\R^n)}\nonumber\\
\lesssim
&2^{4ks_{0}}\sum_{j=0}^{\infty}\norm{\F^{-1}_{x}\varphi_{j}(h_{\eps}(\xi))\varphi_{k}(\xi)\mathscr{F}_{x}u_{0}}_{L^{2}_{x}(\R^n)}.\label{2*}
\end{align}
Combining \eqref{mfp} with \eqref{convolution}, together with
Minkowski's inequality, we have
\begin{align}
&\norm{\langle \partial_t \rangle^{s_0}S_{\eps}(t)u_0}_{L^{2^{*}}(\R,B^{ns_{0}}_{2^{*},2^{*}}(\R^n))}\nonumber\\
\lesssim&\sum_{j=0}^{\infty}\left(\sum_{k=0}^{\infty}2^{(n+4)s_{0}k2^{*}}\norm{\mathscr{F}^{-1}_{x}\varphi_{j}(h_{\eps}(\xi))\varphi_{k}(\xi)\mathscr{F}_{x}u_{0}}^{2^*}_{L^{2}_{x}(\R^n)}\right)^{1/2^{*}}\nonumber\\
\lesssim&\sum_{j=0}^{\infty}\norm{\mathscr{F}^{-1}_{x}\varphi_{j}(h_{\eps}(\xi))\mathscr{F}
u_{0}}_{B^{(n+4)s_{0}}_{2,2^{*}}}.\label{le1}
\end{align}
In view of $H^{(n+4)s_0}\subset B^{(n+4)s_{0}}_{2,2^{*}}$ and
H\"{o}lder's inequality, we have for any $\rho>0$,
\begin{align}
\sum_{j=0}^{\infty}\norm{\mathscr{F}^{-1}_{x}\varphi_{j}(h_{\eps}(\xi))\mathscr{F}
u_{0}}_{B^{(n+4)s_{0}}_{2,2^{*}}}\lesssim
&\sum_{j=0}^{\infty}\norm{\mathscr{F}^{-1}_{x}\varphi_{j}(h_{\eps}(\xi))\mathscr{F}
u_{0}}_{H^{(n+4)s_0}}\nonumber\\
\lesssim
&\left(\sum_{j=0}^{\infty}2^{4j\rho}\norm{\mathscr{F}^{-1}_{x}\varphi_{j}(h_{\eps}(\xi))\mathscr{F}
u_{0}}^{2}_{H^{(n+4)s_0}}\right)^{1/2}.\label{le2}
\end{align}
By Plancherel's identity and the fact that $\supp
\varphi_{j}(h_{\eps}(\xi))\subset\{\xi:h_{\eps}(\xi)\in[2^{j-1},2^{j+1}]\}$,
we easily see that
\begin{eqnarray}
\left(\sum_{j=0}^{\infty}2^{4j\rho}\norm{\mathscr{F}^{-1}_{x}\varphi_{j}(h_{\eps}(\xi))\mathscr{F}
u_{0}}^{2}_{H^{(n+4)s_0}}\right)^{1/2}&\lesssim&
\left(\sum_{j=0}^{\infty}\norm{\langle
h_{\eps}(\xi)\rangle^{\rho}\varphi_{j}(h_{\eps}(\xi))\mathscr{F}
u_{0}}^{2}_{H^{(n+4)s_0}}\right)^{1/2}\nonumber\\
&\lesssim&\left(\sum_{j=0}^{\infty}\norm{\varphi_{j}(h_{\eps}(\xi))\mathscr{F}
u_{0}}^{2}_{H^{(n+4)s_{0}+4\rho}}\right)^{1/2}\nonumber\\
&\lesssim &\norm{u_0}_{H^{(n+4)s_{0}+4\rho}}.\label{le3}
\end{eqnarray}
Taking $s_0$ such that $(n+4)s_{0}+4\rho <s$, from
\eqref{2*}-\eqref{le3} we have the result, as desired.
\end{proof}

By the sharp inclusion $H^{s}\subset L^{\infty}$ for $s>n/2$, it is
obvious that Proposition \ref{p25} is optimal in the sense that it
does not hold for $s=n/2$.

Using the ideas as in Lemma \ref{l23} and Proposition \ref{p25}, we
can show
\begin{proposition}\label{p29}
Let $2^{*}\leq p,q,r\leq\infty$, $s_{0}>1/2^{*}-1/q$ and
$s_{1}>n(1/2^{*}-1/r)$. Then we have
\begin{equation}
\left(\sum_{\alpha\in
\mathbb{Z}^n}\norm{S_{\eps}(t)u_0}^{p}_{L^{q}(\R,L^{r}(Q_{\alpha}))}\right)^{1/p}\lesssim\norm{u_0}_{H^{s_{1}+4s_{0}}}.\label{pqr1}
\end{equation}
In particular, for any $q,p\geq 2^{*}$, $s> n/2-4/q$, it holds
\begin{equation*}
\left(\sum_{\alpha\in
\mathbb{Z}^n}\norm{S_{\eps}(t)u_0}^{p}_{L^{q}(\R,L^{\infty}(Q_{\alpha}))}\right)^{1/p}\lesssim\norm{u_0}_{H^{s}}.
\end{equation*}
\end{proposition}
\begin{proof}[Sketch of the Proof]
 In view of $\ell^{2^*}\subset\ell^{p}$, it suffices to consider the
case $p=2^*$. Using the inclusions $H^{s_{0}}_{p}(\R)\subset
L^{q}(\R)$ and $B^{s_{1}}_{p,p}(\R^n)\subset L^{r}(\R^n)$, we have
\begin{equation}
\norm{u}_{L^{q}(\R,L^{r}(Q_{\alpha}))}
\lesssim\norm{(I-\partial^{2}_{t})^{s_{0}/2}\sigma_{\alpha}u}_{L^{p}(\R,B^{s_{1}}_{p,p}(\R^n))}.\label{qr1}
\end{equation}
Using the same way as in Lemma \ref{l23}, we can show that
\begin{equation}
\left(\sum_{\alpha\in\mathbb{Z}^n}\norm{u}^{p}_{L^{q}(\R,L^{r}(Q_{\alpha}))}\right)^{1/p}
\lesssim\norm{(I-\partial^{2}_{t})^{s_{0}/2}u}_{L^{p}(\R,B^{s_{1}}_{p,p}(\R^n))}.\label{pqr2}
\end{equation}
One can repeat the procedures as in the proof of Lemma \ref{l23} to
conclude that
\begin{equation}
\sum_{\alpha\in\mathbb{Z}^n}\norm{(I-\partial^{2}_{t})^{s_{0}/2}\sigma_{\alpha}S_{\eps}(t)u}^{p}_{L^{p}(\R,B^{s_{1}}_{p,p}(\R^n))}
\lesssim\sum_{j=0}^{\infty}\norm{\mathscr{F}^{-1}\varphi_{j}(h_{\eps}(\xi))\mathscr{F}
u_{0}}_{H^{s_{1}+4s_{0}}(\R^n)}.\label{ppp}
\end{equation}
Applying an analogous way as in the proof of Proposition \ref{p25},
we can get
\begin{equation}
\sum_{j=0}^{\infty}\norm{\mathscr{F}^{-1}\varphi_{j}(h_{\eps}(\xi))\mathscr{F}
u_{0}}_{H^{s_{1}+4s_{0}}(\R^n)}\lesssim\norm{u_{0}}_{H^{s_{1}+4s_{0}+4\rho}}.\label{sum1}
\end{equation}
Combining \eqref{pqr2} with \eqref{sum1}, we immediately get
\eqref{pqr1}.
\end{proof}

\section{Global-local estimates on time-space}
\subsection{Time-global and space-local Strichartz estimates}

We need to make some modifications to the Strichartz estimates,
which are global in the time variable and local in spatial
variables. We always denote by $S_{\eps}(t)$ and
$\mathscr{A}_{\eps}$ the fourth order Schr\"{o}dinger semi-group and
the integral operator as before.
\begin{proposition}\label{p31}
For $n\geq5$, we have
\begin{align}
\sup_{\alpha\in
\mathbb{Z}^n}\norm{S_{\eps}(t)u_0}_{L^{2}_{t,x}(\R\times
Q_{\alpha})}\lesssim &\norm{u_0}_{2},\label{str5}\\
\sup_{\alpha\in\mathbb{Z}^n}\norm{\mathscr{A}_{\eps}F}_{L^{2}_{t,x}(\R\times
Q_{\alpha})}\lesssim &\sum_{\alpha\in
\mathbb{Z}^n}\norm{F}_{L^{1}_{t}L^{2}_{x}(\R\times
Q_{\alpha})},\label{str6}\\
\sup_{\alpha\in\mathbb{Z}^n}\norm{\mathscr{A}_{\eps}F}_{L^{2}_{t,x}(\R\times
Q_{\alpha})}\lesssim &\sum_{\alpha\in
\mathbb{Z}^n}\norm{F}_{L^{2}_{t}L^{2}_{x}(\R\times
Q_{\alpha})}.\label{str7}
\end{align}
\end{proposition}

\begin{proof}
 In view of H\"{o}lder's
inequality and the endpoint Strichartz estimate, we have
\begin{align*}
\norm{S_{\eps}(t)u_0}_{L^{2}_{t,x}(\R\times Q_{\alpha})}\lesssim
&\norm{S_{\eps}(t)u_0}_{L^{2}_{t}L^{2n/(n-4)}_{x}(\R\times
Q_{\alpha})}\nonumber\\
\lesssim &\norm{S_{\eps}(t)u_0}_{L^{2}_{t}L^{2n/(n-4)}_{x}(\R\times
\R^n)}\nonumber\\
\lesssim &\norm{u_0}_{L^{2}_{x}(\R^n)}.
\end{align*}
Using the above ideas and the following Strichartz estimate
\begin{align*}
\norm{\mathscr{A}_{\eps}F}_{L^{2}_{t}L^{2n/(n-4)}_{x}(\R\times
\R^n)}\lesssim
&\norm{F}_{L^{1}_{t}L^{2}_{x}(\R\times \R^{n})},\\
\norm{\mathscr{A}_{\eps}F}_{L^{2}_{t}L^{2n/(n-4)}_{x}(\R\times
\R^n)}\lesssim &\norm{F}_{L^{2}_{t}L^{2n/(n+4)}_{x}(\R\times
\R^{n})},
\end{align*}
 one can
easily get \eqref{str6} and \eqref{str7}.
\end{proof}
\begin{remark}\label{rm} Using Proposition \ref{Guo1}, we can verify that Proposition \ref{p31} also holds for
$n=3,4$ when only consider the case $\eps=1$. For example, taking
$q=2, p=8$ in Proposition \ref{Guo1} when $n=3$ we can get the
result desired.
\end{remark}
Since the endpoint Strichartz estimates are used in the
above proof , Proposition \ref{p31} only holds for $n\geq5$. It is
not clear for us whether \eqref{str5}  holds or not
 for $n=3,4$ when consider $S_{0}(t)$. It is why we need an additional condition that
 $u_{0}\in \dot{H}^{-3/2}$ is small in the case $n=3,4$ . However, we have the
 following (cf. \cite{KPV2})
\begin{proposition}\label{p32}
Let $n=3,4$. Then we have for any $1\leq r<6/5 $ for $n=3$, or
$1\leq r<4/3$ for $n=4$
\begin{equation}
\sup_{\alpha\in
\mathbb{Z}^n}\norm{S_{\eps}(t)u_0}_{L^{2}_{t,x}(\R\times
Q_{\alpha})}\lesssim
\min(\norm{(-\Delta)^{-3/4}u_{0}}_{L^2(\R^n)},\norm{u_0}_{L^{2}\cap
L^{r}(\R^n)}).\label{s2}
\end{equation}
Observing  \eqref{s2} is strictly weaker than \eqref{str5} in the
low frequencies.
\end{proposition}

\begin{proof}
By Remark \ref{rm} and Lemma \ref{l34}, it suffices to show that
\begin{equation}
\sup_{\alpha\in
\mathbb{Z}^n}\norm{S_{0}(t)u_0}_{L^{2}_{t,x}(\R\times
Q_{\alpha})}\lesssim \norm{u_0}_{L^{2}\cap L^{r}(\R^n)}.\label{s22}
\end{equation}
Using the unitary property in $L^2$ and the $L^{p}$-$L^{p'}$ decay
estimate of $S_{0}(t)$, we have
\begin{equation}
\sup_{\alpha\in \mathbb{Z}^n}\norm{S_{0}(t)u_0}_{L^{2}_{x}(
Q_{\alpha})}\lesssim
(1+|t|)^{\frac{n}{4}(1-\frac{2}{r})}\norm{u_0}_{L^{2}\cap
L^{r}(\R^n)}.\label{s2r}
\end{equation}
Taking the $L^{2}_{t}$ norm in both sides of \eqref{s2r}, we
immediately get \eqref{s22}.
\end{proof}

\begin{proposition}\label{p33}
Let $n=3,4$. Then we have
\begin{equation*}
\sup_{\alpha\in\mathbb{Z}^n}\norm{\mathscr{A}_{\eps}F}_{L^{2}_{t,x}(\R\times
Q_{\alpha})}\lesssim\sum_{\alpha\in
\mathbb{Z}^n}\norm{F}_{L^{1}_{t}L^{2}_{x}(\R\times Q_{\alpha})}.
\end{equation*}
\end{proposition}

\begin{proof}

 Firstly, the case $\eps=1$ holds by Remark \ref{rm}. Secondly, we notice that
\begin{equation*}
\sup_{\alpha\in \mathbb{Z}^n}\norm{S_{0}(t)u_0}_{L^{2}_{x}(
Q_{\alpha})}\lesssim (1+|t|)^{-3/4}\norm{u_0}_{L^{2}\cap
L^{r}(\R^n)}.
\end{equation*}
It follows that
\begin{equation*}
\norm{\mathscr{A}_{0}F}_{L^{2}_{x}( Q_{\alpha})}\lesssim
\int_{\R}(1+|t|)^{-3/4}\norm{F(\tau)}_{L^{1}_{x}\cap
L^{2}_{x}(\R^n)}d\tau.
\end{equation*}
Using Young's inequality, one has that
\begin{equation}
\norm{\mathscr{A}_{0}F}_{L^{2}_{t,x}(\R\times Q_{\alpha})}\lesssim
\norm{F}_{L^{1}(\R, L^{1}_{x}\cap L^{2}_{x}(\R^n))}.\label{2211}
\end{equation}
In view of H\"{o}lder's inequality, \eqref{2211} yields the result,
as desired.
\end{proof}

\begin{remark}
For $n=2$, we can't establish the similar result as in Proposition
\ref{p33}. So our results can't cover this case.
\end{remark}

\subsection{Note on  the time-global and space-local smoothing effects}
Kenig, Ponce and Vega \cite{KPV1,KPV2} obtained the local
smoothing effect estimates for the Schr\"{o}dinger semi-group
$e^{it\Delta}$, and their results can also be developed to the
fourth order Schr\"{o}dinger semi-group
$e^{it(\Delta^2-\eps\Delta)}$. On the basis of their results and
Proposition \ref{p31}, we can obtain a time-global version of the
local smoothing effect estimates with the inhomogeneous
differential operator $(I-\Delta)^{1/2}$ instead of the
homogeneous one $\nabla$, which is useful to control the low
frequencies parts of the nonlinearity.

\begin{lemma}[cf. \cite{KPV1}]\label{l34}
Let $\Omega$ be an open set in $\R^n$, $\phi$ be a $C^{1}(\Omega)$
function such that $\nabla\phi(\xi)\neq 0$ for any $\xi\in \Omega$.
Assume that there is a $N\in\N$ such that for any
$\bar{\xi}:=(\xi_{1},...,\xi_{n-1})\in\R^{n-1}$ and $r\in \R$, the
equation $\phi(\xi_{1},...,\xi_{k},x,\xi_{k+1},...,\xi_{n-1})=r$ has
at most $N$ solutions. For $a(x,s)\in L^{\infty}(\R^{n}\times \R)$
and $f\in\mathscr{S}'(\R^{n})$, we denote
\begin{equation*}
W(t)f(x)=\int_{\Omega}e^{i(t\phi(\xi)+x\xi)}a(x,\phi(\xi))\hat{f}(\xi)d\xi,
\end{equation*}
then for $n\geq 2$, we have
\begin{equation}
\norm{W(t)f(x)}_{L^{2}_{t,x}(\R\times B(0,R))}\leq
CNR^{1/2}\norm{|\nabla\phi|^{-1/2}\hat{f}}_{L^{2}(\Omega)}.\label{222}
\end{equation}
\end{lemma}

\begin{corollary}\label{c35}
Let $n\geq 5$, $S_{\eps}(t)=e^{it(\Delta^2-\eps\Delta)}$. We have
\begin{align}
\sup_{\alpha\in\mathbb{Z}^n}\norm{S_{\eps}(t)u_0}_{L^{2}_{t,x}(\R\times
Q_{\alpha})}\lesssim  &\norm{u_0}_{H^{-3/2}},\label{223} \\
\norm{\mathscr{A}_{\eps}f}_{L^{\infty}(\R,H^{3/2})} \lesssim
&\sum_{\alpha\in\mathbb{Z}^n}\norm{f}_{L^{2}_{t,x}(\R\times
Q_{\alpha})}.\label{I4}
\end{align}
For $n=2,3,4$, \eqref{I4} also holds if one substitutes $H^{3/2}$ by
$\dot{H}^{3/2}$.
\end{corollary}

\begin{proof}
Let $\Omega=\R^{n}\setminus B(0,1)$,
$\phi(\xi)=\abs{\xi}^{4}+\eps\abs{\xi}^{2}$ and $\psi$ be as in
\eqref{cutoff}, $a(x,s)=1-\psi(s)$ in Lemma \ref{l34}. Taking
$W(t):=S_{\eps}(t)\mathscr{F}^{-1}(1-\psi)\mathscr{F}$, from
\eqref{222} we have
\begin{equation}
\sup_{\alpha\in\mathbb{Z}^n}\norm{S_{\eps}(t)\mathscr{F}^{-1}(1-\psi)\mathscr{F}
u_0}_{L^{2}_{t,x}(\R\times Q_{\alpha})}\lesssim
\norm{\abs{\xi}^{-3/2}u_0}_{L^{2}_{\xi}(\R^{n}\setminus
B(0,1))}.\label{224}
\end{equation}
Using Proposition \ref{p31}, we have
\begin{align}
\norm{S_{\eps}(t)\mathscr{F}^{-1}\psi\mathscr{F}
u_0}_{L^{2}_{t,x}(\R\times Q_{\alpha})}\lesssim
&\norm{\mathscr{F}^{-1}\psi\mathscr{F}
u_0}_{L^{2}_{x}(\R^n)}
\lesssim \norm{\hat{u}_0}_{L^{2}_{\xi}(B(0,2))}.\label{225}
\end{align}
Combining \eqref{224} with \eqref{225} we have \eqref{223},  as
desired. \eqref{I4} is the dual version of \eqref{223}.
\end{proof}

Using the method of Kenig, Ponce and Vega \cite{KPV2} and a
modification of Hao, Hsiao and Wang \cite{Hao1}, we can prove the
following local smoothing effect estimates for the inhomogeneous
part of the solutions of the fourth order Schr\"{o}dinger equation.

 Now we turn to consider
the inhomogeneous Cauchy problem:
\begin{align}
    i\partial_t u&=\Delta^2 u-\eps\Delta u+F(t,x),\quad t\in\Real,\, x\in \Real^n,\label{eq2.1}\\
    u(0,x)&=0, \label{eq2.2}
\end{align}
with $F\in\mathscr{S}(\Real\times\Real^n)$ and $\eps=0, 1$. We have
the following estimate on the local smoothing effects in this
inhomogeneous case.

\begin{proposition}[Local smoothing effect: inhomogeneous
case]\label{p36} For any multi-index $\nu, \abs{\nu}=3$, the
solution $u(t,x)$ of the Cauchy problem \eqref{eq2.1}-\eqref{eq2.2}
satisfies
\begin{align}\label{lse3}
     \sup_{\alpha\in \Z^n}\norm{D^{\nu}
    u(t,x)}_{L^{2}_{t,x}(\R\times Q_{\alpha})
    }\lesssim\sum_{\alpha\in \Z^n}\norm{F}_{L^{2}_{t,x}(\R\times Q_{\alpha})}.
\end{align}
\end{proposition}

\begin{proof}Separating
\begin{align*}
F&=\sum_{\alpha\in \Z^n}F \chi_{Q_\alpha}=\sum_{\alpha\in
\Z^n}F_\alpha,
 \intertext{and}
  u&=\sum_{\alpha\in \Z^n}u_\alpha,
\end{align*}
where $u_\alpha(t,x)$ is the corresponding solution of the Cauchy
problem
\begin{align}
    i\partial_t u_\alpha&=\Delta^2 u_\alpha-\eps\Delta u_\alpha+F_\alpha(t,x),\quad t\in\Real,\, x\in \Real^n,\label{eq21}\\
    u_\alpha(0,x)&=0\label{eq22},
\end{align}
we formally take Fourier transform in both variables $t$ and $x$ in
the equation \eqref{eq21} and obtain
\begin{equation*}
    \hat{u}_\alpha(\tau,\xi)=\frac{\hat{F}_\alpha(\tau,\xi)}{\tau-\eps\abs{\xi}^2-\abs{\xi}^4},\quad
    \textrm{ for each } \alpha\in\Z^n.
\end{equation*}

By the Plancherel theorem in the time variable, we can get
\begin{align}
    &\sup_{\alpha\in \Z^n}\norm{D^{\nu}
    u_\beta(t,x)}_{L^{2}_{\tau,x}(\R\times Q_{\alpha})}
    =\sup_{\alpha\in \Z^n}\norm{D^{\nu}
    \F_{t} u_{\beta}(\tau,x)}_{L^{2}_{\tau,x}(\R\times Q_{\alpha})}\nonumber\\
    =&\sup_{\alpha\in \Z^n}\norm{D^{\nu}
    \ft_{\xi}^{-1}\left(\frac{\hat{F}_\beta(\tau,\xi)}{\tau-\eps\abs{\xi}^2-\abs{\xi}^4}\right)}_{L^{2}_{\tau,x}(\R\times Q_{\alpha})}\nonumber\\
    =& \sup_{\alpha\in \Z^n}\norm{
    \F_{\xi}^{-1}\left(\frac{\xi^{\nu}}{\tau-\eps\abs{\xi}^2-\abs{\xi}^4}\hat{F}_\beta(\tau,\xi)\right)}_{L^{2}_{\tau,x}(\R\times Q_{\alpha})}\nonumber\\
   =& \sup_{\alpha\in \Z^n}\left(\int_{Q_\alpha}\int_{-\infty}^\infty \abs{\F_{\xi}^{-1}
    \left(\frac{\xi^\nu}{\tau-\eps\abs{\xi}^2-\abs{\xi}^4}\hat{F}_\beta(\tau,\xi)\right)}^{2}d\tau dx\right)^{1/2}. \label{eq2.3}
\end{align}
In order to continue the above estimate, we introduce the following
estimate.
\begin{lemma}\label{l37}
Let $\M f=\FF m(\xi)\F f$ and
$m(\xi)=\frac{\xi^\nu}{1-\eps\tau^{-\frac{1}{2}}\abs{\xi}^2-\abs{\xi}^4}$
for $\tau>0$ and $\eps\in\{0,1\},\abs{\nu}=3$, where $\F(\FF)$
denotes the Fourier (inverse, respectively) transform in x only.
Then, we have
\begin{equation*}
\sup_{\alpha\in \Z^n} \left(\int_{Q_\alpha}\abs{\M
(g\chi_{Q_\beta})}^{2}dx\right)^{1/2}\leq CR
\left(\int_{Q_\beta}\abs{g}^{2}dx\right)^{1/2}.
\end{equation*}
where $R$ is the  size of $Q_{\alpha}$.
\end{lemma}

\begin{proof}
Since the proof is similar to Hao, Hsiao and Wang \cite{Hao1}, we
omit it.
\end{proof}

Now, we go back to the proof of Proposition \ref{p36}. We first
consider the part when $\tau>0$ in \eqref{eq2.3}, \textit{i.e.},
\begin{align*}
    &\sup_{\alpha\in \Z^n}\left(\int_{Q_\alpha}\int_0^\infty \abs{\F_{\xi}^{-1}
    \left(\frac{\xi^\nu}{\tau-\eps\abs{\xi}^2-\abs{\xi}^4}\hat{F}_\beta(\tau,\xi)\right)}^{2}d\tau d
    x\right)^{1/2}\\
    =&\left(\int_0^\infty \tau^{\frac{n}{4}-\frac{1}{2}} \sup_{\alpha\in
    \Z^n}\int_{\tau^{\frac{1}{4}}Q_\alpha}\abs{\int_{\Real^n}
    e^{iy\eta}\frac{\eta^\nu}{1-\eps\tau^{-\frac{1}{2}}\abs{\eta}^2-\abs{\eta}^4}\hat{F}_\beta
    (\tau,\tau^{\frac{1}{4}}\eta)d\eta}^{2}dy d\tau\right)^{1/2}\\
    =&\left(\int_0^\infty \tau^{-\frac{n}{4}-\frac{1}{2}} \sup_{\alpha\in
    \Z^n}\int_{\tau^{\frac{1}{4}}Q_\alpha}\abs{\M \F_t
    F_\beta(\tau,\tau^{-\frac{1}{4}}y)}^{2}dy d\tau\right)^{1/2}\\
    \lesssim &\left(\int_0^\infty \tau^{-\frac{n}{4}}
    \int_{\tau^{\frac{1}{4}}Q_\beta} \abs{\F_t F_\beta(\tau,\tau^{-\frac{1}{4}}y)}^{2}d
    y d\tau\right)^{1/2}\\
    \lesssim &\left(\int_0^\infty
    \int_{Q_\beta} \abs{\F_t F(\tau,x)}^{2}d
    x d\tau\right)^{1/2}
    \end{align*}
where we have used the changes of variables, Lemma \ref{l37} and the
identity
\begin{align*}
    (\F_\eta^{-1}\hat{F}_\beta(\tau,\tau^{\frac{1}{4}}\eta))(\tau,y)=\tau^{-\frac{n}{4}}\F_t
    F_\beta(\tau,\tau^{-\frac{1}{4}}y).
\end{align*}

For the part when $\tau\in(-\infty, 0)$ in \eqref{eq2.3},  it is
easier to handle since this  corresponds to the symbol
$\frac{\eta^\nu}{1+\eps\abs{\tau}^{-\frac{1}{2}}\abs{\eta}^2+\abs{\eta}^4}$,
which has no singularity.

Therefore,  we obtain, by the Plancherel theorem, the Sobolev
embedding theorem and the H\"older's inequality
\begin{align*}
   & \sup_{\alpha\in \Z^n}\norm{D^{\nu}
    u_\beta(t,x)}_{L^{2}_{t,x}(\R\times Q_{\alpha})}\\
    \lesssim & \left(\int_{-\infty}^\infty \int_{Q_\beta} \abs{\F_t F(\tau,x)}^{2}d
    x d\tau\right)^{1/2}
   \lesssim  \left(\int_{Q_\beta}\int_{-\infty}^\infty \abs{\F_t F(\tau,x)}^{2} d\tau d
    x\right)^{1/2}\\
    \lesssim & \left(\int_{Q_\beta}\norm{ F(\cdot,x)}_{L^{2}}^{2}d
    x\right)^{1/2}
  \lesssim \norm{F}_{L^{2}_{t,x}(\R\times Q_{\beta})},
\end{align*}
which implies the desired result \eqref{lse3}. In general, the
solution $u(t,x)$ of \eqref{eq2.1} may not vanish at $t=0$,  we can
dealt with this case by the method of Hao, Hsiao and Wang
\cite{Hao}.
\end{proof}

\section{Proof of Theorem 1.1}
\begin{lemma}\label{l41}
(Sobolev Inequality). Let $\Omega\subset\R^{n}$ be a bounded domain
with $\partial\Omega\in C^{m}$, $m, \ell\in\mathbb{N}\cup\{0\}$,
$1\leq r,p,q\leq\infty$. Assume that
$$\frac{\ell}{m}\leq\theta\leq1,\quad \frac{1}{p}-\frac{\ell}{n}=\theta(\frac{1}{r}-\frac{m}{n})+\frac{1-\theta}{q}.$$
Then we have
\begin{equation*}
\sum_{|\beta|=\ell}\norm{D^{\beta}u}_{L^{p}(\Omega)}\lesssim
\norm{u}^{1-\theta}_{L^{q}(\Omega)}\norm{u}^{\theta}_{W^{m}_{r}(\Omega)},
\end{equation*}
where $\norm{u}_{W^{m}_{r}(\Omega)}=\sum\limits_{|\beta|\leq
m}\norm{D^{\beta}u}_{L^{r}(\Omega)}$.

\end{lemma}

\begin{proof}[Proof of Theorem 1.1]
For simplicity, we first consider a simple case
$s=[\frac{n}{2}]+\frac{11}{2}$ and there is no difficult to
generalize the proof to the case $s>\frac{n}{2}+\frac{9}{2}$ with
$s+1/2\in \mathbb{N}$. We assume without loss of generality that
\begin{align*}
&F((\partial_x^\alpha u)_{\abs{\alpha}\leq 3},
(\partial_x^\alpha\bar{u})_{\abs{\alpha}\leq 3}
)\\
=&F((\partial_x^\alpha u)_{\abs{\alpha}\leq3})\\
=&\sum_{\Lambda_{R_{0},R_{1},R_{2},R_{3}}}C_{R_{0}R_{1}R_{2}R_{3}}u^{R_{0}}(D^{\alpha_{1}}
u)^{R_{1}}(D^{\alpha_{2}}u)^{R_{2}}(D^{\alpha_{3}}u)^{R_{3}},
\end{align*}
where $R_{i}, \alpha_{i}$ ($i=1,2,3$) are multi-indices,
$\abs{\alpha_{i}}=i$ $(i=1,2,3)$ and
$$\Lambda_{R_{0},R_{1},R_{2},R_{3}}=\{(R_{0},R_{1},R_{2},R_{3}):m+1\leq R_{0}+|R_{1}|+|R_{2}|+|R_{3}|\leq M+1\}.$$
Since we only use the Sobolev norm, $u$ and $\bar{u}$ have the same
norm. Hence, the general cases can be handled in the same way.
Denote
\begin{align*}
\lambda_{1}(v):=&\norm{v}_{\ell^{\infty}_{\alpha}(L^{2}_{t,x}(\R\times
Q_{\alpha}))},\\
\lambda_{2}(v):=&\norm{v}_{\ell^{2^*}_{\alpha}(L^{\infty}_{t,x}(\R\times
Q_{\alpha}))},\\
\lambda_{3}(v):=&\norm{v}_{\ell^{2^*}_{\alpha}(L^{2m}_{t}L^{\infty}_{x}(\R\times
Q_{\alpha}))},
 \end{align*}
 Put
\begin{equation*}
{\mathscr{D}_{n}}=\left\{{ u:\sum_{|\beta|\leq\
[n/2]+7}\lambda_{1}(D^{\beta}u)+\sum_{|\beta|\leq3}\sum_{i=2,3}\lambda_{i}(D^{\beta}u)}\leq\varrho
\right\}.
\end{equation*}
Considering the mapping
\begin{equation}
\mathscr{\mathscr{T}}:u(t)\mapsto
S_{\eps}(t)u_{0}-i\mathscr{A}_{\eps}F((\partial_x^\alpha
u)_{\abs{\alpha}\leq 3}),\label{ct}
\end{equation}
we show that $\mathscr{T}:{\mathscr{D}_{n}}\rightarrow
{\mathscr{D}_{n}}$ is a contraction mapping for any $n\geq 3$.

\emph{Step 1.} For any $u\in \mathscr{D}_{n}$, we first estimate
$\lambda_{1}(D^{\beta}\mathscr{T}u)$ for
$|\beta|\leq7+[\frac{n}{2}]$. Using Lemma \ref{l41}, the cases
$|\beta|=1$ and $|\beta|=2$ can be dominated by terms of the cases
$|\beta|=0$ and $|\beta|=3$. Thus, we only need to consider the
cases $|\beta|=0$ and $|\beta|=3$. We split it into
 three cases as the following.

\emph{Case 1.} $n\geq3$, $3\leq |\beta|\leq7+[\frac{n}{2}]$. For
simplicity, we denote
$\Lambda_{\tilde{R}}=\Lambda_{R_{0},R_{1},R_{2},R_{3}}$. In view
of Corollary \ref{c35} and Proposition \ref{p36}, we have for any
$\beta$ with $3\leq |\beta|\leq7+[\frac{n}{2}]$ that
\begin{align}
&\lambda_{1}(D^{\beta}\mathscr{T}u)\nonumber\\
\lesssim
&\norm{S_{\eps}(t)D^{\beta}u_0}_{\ell^{\infty}_{\alpha}(L^{2}_{t,x}(\R\times
Q_{\alpha}))}\nonumber\\
&+\sum_{|\beta|\leq4+[\frac{n}{2}]}\sum_{\Lambda_{\tilde{R}}}\sum_{\alpha\in\mathbb{Z}^{n}}
\norm{D^{\beta}[u^{R_{0}}(D^{\alpha_{1}}
u)^{R_{1}}(D^{\alpha_{2}}u)^{R_{2}}(D^{\alpha_{3}}u)^{R_{3}}]}_{L^{2}_{t,x}(\R\times
Q_{\alpha})}\nonumber\\
\lesssim &
\norm{u_{0}}_{H^s}+\sum_{|\beta|\leq4+[\frac{n}{2}]}\sum_{\Lambda_{\tilde{R}}}\sum_{\alpha\in\mathbb{Z}^{n}}\norm{D^{\beta}[u^{R_{0}}(D^{\alpha_{1}}
u)^{R_{1}}(D^{\alpha_{2}}u)^{R_{2}}(D^{\alpha_{3}}u)^{R_{3}}]}_{L^{2}_{t,x}(\R\times
Q_{\alpha})}.\label{ct1}
\end{align}
For simplicity, we only consider the case
$u^{R_{0}}(D^{\alpha_{1}}
u)^{R_{1}}(D^{\alpha_{2}}u)^{R_{2}}(D^{\alpha_{3}}u)^{R_{3}}=u^{R_{0}}(\partial^{3}_{x_{1}}u)^{R_{3}}$
in \eqref{ct1} and the general case can be treated in an analogous
way\footnote{One can see below for a general treating.}. So, one
can rewrite \eqref{ct1} as
\begin{equation*}
\sum_{3\leq|\beta|\leq7+[\frac{n}{2}]}\lambda_{1}(D^{\beta}\mathscr{T}u)\lesssim
\norm{u_{0}}_{H^s}+\sum_{|\beta|\leq4+[\frac{n}{2}]}\sum_{\Lambda_{\tilde{R}}}\sum_{\alpha\in\mathbb{Z}^{n}}\norm{D^{\beta}[u^{R_{0}}(\partial^{3}_{x_{1}}u)^{R_{3}}]}_{L^{2}_{t,x}(\R\times
Q_{\alpha})}.
\end{equation*}

Observing that
\begin{equation}
|D^{\beta}[u^{R_{0}}(\partial^{3}_{x_{1}}u)^{R_{3}}]|\lesssim
\sum_{\beta_{1}+\cdots+\beta_{\tilde{R}}
=\beta}|D^{\beta_{1}}u\cdots
D^{\beta_{R}}uD^{\beta_{R_{0}+1}}(\partial^{3}_{x_{1}}u)D^{\beta_{\tilde{R}}}(\partial^{3}_{x_{1}}u)|.\label{dd1}
\end{equation}
By H\"{o}lder's inequality, we get
\begin{equation}
\norm{D^{\beta}[u^{R_{0}}(\partial^{3}_{x_{1}}u)^{R_{3}}]}_{L^{2}_{x}(Q_{\alpha})}\lesssim
\sum_{\beta_{1}+...+\beta_{\tilde{R}}}\prod_{i=1}^{R_{0}}\norm{D^{\beta_{i}}u}_{L^{p_{i}}(Q_{\alpha})}
\prod_{i=R_{0}+1}^{\tilde{R}}\norm{D^{\beta_{i}}(\partial^{3}_{x_{1}}u)}_{L^{p_{i}}(Q_{\alpha})},\label{dd2}
\end{equation}
where
\begin{align*}
p_{i}=\left\{
\begin{array}{ll}
2|\beta|/|\beta_{i}|,&|\beta_{i}|\geq1,\\
\infty, &|\beta_{i}|=0.
\end{array}
\right.
\end{align*}
It is clear that for $\theta_{i}=|\beta_{i}|/|\beta|$,
$$\frac{1}{p_{i}}-\frac{|\beta_{i}|}{n}=\theta_{i}(\frac{1}{2}-\frac{|\beta|}{n})+\frac{1-\theta_{i}}{\infty}.$$
Using  Sobolev's inequality, one has that for
$B_{\alpha}=\{x:|x-a|\leq\sqrt{n}\}$
\begin{equation}
\norm{D^{\beta_{i}}u}_{L^{p_{i}}_{x}(Q_{\alpha})}\leq\norm{D^{\beta_{i}}u}_{L^{p_{i}}_{x}(B_{\alpha})}\lesssim
\norm{u}^{1-\theta_{i}}_{L^{\infty}_{x}(B_{\alpha})}\norm{u}^{\theta_{i}}_{W^{|\beta|}_{2}(B_{\alpha})},\quad
i=1,\cdots,R_{0},\label{sb1}
\end{equation}
\begin{equation}
\norm{D^{\beta_{i}}(\partial^{3}_{x_{1}}u)}_{L^{p_{i}}_{x}(Q_{\alpha})}\lesssim
\norm{\partial^{3}_{x_{1}}u}^{1-\theta_{i}}_{L^{\infty}_{x}(B_{\alpha})}
\norm{\partial^{3}_{x_{1}}u}^{\theta_{i}}_{W^{|\beta|}_{2}(B_{\alpha})},\quad
i=R_{0}+1,\cdots,\tilde{R}.\label{sb2}
\end{equation}
Since
$$\sum_{i=1}^{\tilde{R}}\theta_{i}=1,\quad \sum_{i=1}^{\tilde{R}}(1-\theta_{i})=\tilde{R}-1,$$
by \eqref{dd1}-\eqref{sb1} we have
\begin{align}
\norm{D^{\beta}[u^{R_{0}}(\partial^{3}_{x_{1}}u)^{R_{3}}]}_{L^{2}_{x}(Q_{\alpha})}\lesssim
&\sum_{|\beta|\leq4+[\frac{n}{2}]}\left(\norm{u}_{W^{|\beta|}_{2}(B_{\alpha})}+\norm{\partial^{3}_{x_{1}}u}_{W^{|\beta|}_{2}(B_{\alpha})}\right)\nonumber\\
&\times\left(\norm{u}^{\tilde{R}-1}_{L^{\infty}_{x}(B_{\alpha})}+\norm{\partial^{3}_{x_{1}}u}^{\tilde{R}-1}_{L^{\infty}_{x}(B_{\alpha})})\right)\nonumber\\
\lesssim
&\sum_{|\gamma|\leq7+[\frac{n}{2}]}\norm{D^{\gamma}u}_{L^{2}_{x}(B_{\alpha})}\sum_{|\beta|\leq3}
\norm{D^{\beta}(\partial^{3}_{x_{1}}u)}^{\tilde{R}-1}_{L^{\infty}_{x}(B_{\alpha})}.\label{sb3}
\end{align}
It follows, from \eqref{sb2} and
$\ell^{2^*}\subset\ell^{\tilde{R}-1}$, that
\begin{align*}
&\sum_{|\beta|\leq4+[\frac{n}{2}]}\sum_{\alpha\in\mathbb{Z}^n}\norm{D^{\beta}[u^{R_{0}}(\partial^{3}_{x_{1}}u)^{R_{3}}]}_{L^{2}_{t,x}(\R\times
Q_{\alpha})}\\
\lesssim&
\sum_{\alpha\in\mathbb{Z}^n}\sum_{|\gamma|\leq7+[\frac{n}{2}]}\norm{D^{\gamma}u}_{L^{2}_{t,x}(\R\times
B_{\alpha})}\sum_{|\beta|\leq3}\norm{D^{\beta}u}_{L^{\infty}_{t,x}(\R\times
B_{\alpha})}\\
\lesssim&
\sum_{|\gamma|\leq7+[\frac{n}{2}]}\lambda_{1}(D^{\gamma}u)\sum_{|\beta|\leq3}\lambda_{2}(D^{\beta}u)^{\tilde{R}-1}.
\end{align*}
Hence, in view of \eqref{ct1}
 and \eqref{sb3} we have
\begin{equation*}
\sum_{3\leq|\beta|\leq7+[\frac{n}{2}]}\lambda_{1}(D^{\beta}\mathscr{T}u)\lesssim\norm{u_0}_{H^{s}}+\sum_{\tilde{R}=m+1}^{M+1}\varrho^{\tilde{R}}.
\end{equation*}

\emph{Case 2.} $n\geq5$ and $|\beta|=0 $. By Corollary \ref{c35},
the local Strichartz estimate \eqref{str6} and H\"{o}lder's
inequality, we obtain
\begin{align*}
\lambda_{1}(\mathscr{T}u)\lesssim
&\norm{S_{\eps}(t)u_{0}}_{\ell^{\infty}_{\alpha}(L^{2}_{t,x}(\R\times
Q_{\alpha}))}+\norm{\mathscr{A}_{\eps}F((\partial_x^\alpha
u)_{\abs{\alpha}\leq3})}_{\ell^{\infty}_{\alpha}(L^{2}_{t,x}(\R\times
Q_{\alpha}))}\\
\lesssim
&\norm{u_{0}}_{L^2}+\sum_{\alpha\in\mathbb{Z}^{n}}\norm{F((\partial_x^\alpha
u)_{\abs{\alpha}\leq3})}_{L^{1}_{t}L^{2}_{x}(\R\times
Q_{\alpha})}\\
\lesssim
&\norm{u_0}_{L^2}+\sum_{\Lambda_{\tilde{R}}}\sum_{\alpha\in\mathbb{Z}^n}\norm{u^{R_{0}}(\partial^{3}_{x_{1}}u)^{R_{3}}}_{L^{1}_{t}L^{2}_{x}(\R\times
Q_{\alpha})}
\\
 \lesssim
&\norm{u_{0}}_{L^2}+\sum_{\tilde{R}=m+1}^{M+1}\sum_{|\gamma|\leq3}\sup_{\alpha\in\mathbb{Z}^{n}}
\norm{D^{\gamma}u}_{L^{2}_{t,x}(\R\times
Q_{\alpha})}\sum_{|\beta|\leq3}\sum_{\alpha\in\mathbb{Z}^{n}}
\norm{D^{\beta}u}^{\tilde{R}-1}_{L^{2(\tilde{R}-1)}_{t}L^{\infty}_{x}(\R\times
Q_{\alpha})}\\
\lesssim
&\norm{u_{0}}_{L^2}+\sum_{\tilde{R}=m+1}^{M+1}\sum_{|\gamma|\leq3}\lambda_{1}({D^{\gamma}u})\sum_{i=2,3}\sum_{|\beta|\leq3}\lambda_{i}(D^{\beta}u)^{\tilde{R}-1}\\
\lesssim
&\norm{u_0}_{L^2}+\sum_{{\tilde{R}=m+1}}^{M+1}\varrho^{\tilde{R}}.
\end{align*}

\emph{Case 3.} $n=3,4$, $|\beta|=0$. By Propositions \ref{p32}, we
have
\begin{equation*}
\lambda_{1}(\mathscr{T}u)\lesssim\norm{u_0}_{\dot{H}^{-3/2}}+\sum_{\alpha\in\mathbb{Z}^{n}}\norm{F((\partial_x^\alpha
u)_{\abs{\alpha}\leq3})}_{L^{1}_{t}L^{2}_{x}(\R\times Q_{\alpha})}.
\end{equation*}
Using the same way as in Case 2, we have
\begin{equation}
\lambda_{1}(\mathscr{T}u)\lesssim\norm{u_0}_{\dot{H}^{-3/2}}+\sum_{\tilde{R}=m+1}^{M+1}\varrho^{\tilde{R}}.\label{sct1}
\end{equation}

\emph{Step 2.} We consider the estimates of
$\lambda_{2}(D^{\beta}\mathscr{T}u)$ for $|\beta|\leq3$. Using the
estimates of the maximal function as in Proposition \ref{p25}, we
have for $|\beta|\leq3$ and $0<\rho\ll 1$
\begin{align}
\lambda_{2}(D^{\beta}\mathscr{T}u) \lesssim
&\norm{S_{\eps}(t)D^{\beta}u_0}_{\ell^{2^*}_{\alpha}(L^{\infty}_{t,x}(\R\times
Q_{\alpha}))} +\norm{\mathscr{A}_{\eps}D^{\beta}F((\partial_x^\alpha
u)_{\abs{\alpha}\leq3}))}_{\ell^{2^*}_{\alpha}(L^{\infty}_{t,x}(\R\times
Q_{\alpha}))}\nonumber\\
\lesssim
&\norm{D^{\beta}u_0}_{H^{n/2+\rho}}+\sum_{|\beta|\leq3}\norm{D^{\beta}F((\partial_x^\alpha u)_{\abs{\alpha}\leq3})}_{L^{1}(\R,H^{n/2+\rho})}\nonumber\\
\lesssim
&\norm{u_0}_{H^{n/2+3+\rho}}+\sum_{|\beta|\leq4+[n/2]}\sum_{\alpha\in\mathbb{Z}^{n}}\norm{D^{\beta}F((\partial_x^\alpha
u)_{\abs{\alpha}\leq3})}_{L^{1}_{t}L^{2}_{x}(\R\times
Q_{\alpha}))}.\label{ct2}
\end{align}
Applying the same way as in Step 1, for any $|\beta|\leq4+[n/2]$, we
can get
\begin{eqnarray}
\sum_{\alpha\in\mathbb{Z}^{n}}\norm{D^{\beta}F((\partial_x^\alpha
u)_{\abs{\alpha}\leq3})}_{L^{2}_{x}(B_{\alpha})}\lesssim\sum_{\tilde{R}=m+1}^{M+1}\sum_{|\beta|\leq3}\norm{D^{\beta}u}^{\tilde{R}-1}_{L^{\infty}_{x}(B_{\alpha})}\nonumber\\
\times
\sum_{|\gamma|\leq7+[n/2]}\norm{D^{\gamma}u}_{L^{2}_{x}(B_{\alpha})}.\label{sct2}
\end{eqnarray}
By H\"{o}lder's inequality, we have from \eqref{ct2} that
\begin{align}
\norm{D^{\beta}F((\partial_x^\alpha
u)_{\abs{\alpha}\leq3})}_{L^{1}_{t}L^{2}_{x}(\R\times
B_{\alpha})}\lesssim
\sum_{\tilde{R}=m+1}^{M+1}\sum_{\abs{\gamma}\leq7+[n/2]}\norm{D^{\gamma}u}_{L^{2}_{t,x}(\R\times B_{\alpha})}\nonumber \\
\times\sum_{|\beta|\leq3}\norm{D^{\beta}u}^{\tilde{R}-1}_{L^{2(\tilde{R}-1)}_{t}L^{\infty}_{x}(\R\times
B_{\alpha})}.\label{nsct2}
\end{align}
Summating \eqref{sct2} over all $\alpha\in\Z^{n}$, we have for any
$|\beta|\leq4+[n/2]$
\begin{align*}
&\sum_{\alpha\in\mathbb{Z}^{n}}\norm{D^{\beta}F((\partial_x^\alpha u)_{\abs{\alpha}\leq3})}_{L^{1}_{t}L^{2}_{x}(\R\times Q_{\alpha})}\\
\lesssim
&\sum_{\tilde{R}=m+1}^{M+1}\sum_{|\gamma|\leq7+[n/2]}\lambda_{1}(D^{\gamma}u)\sum_{|\beta|\leq3}\sum_{\alpha\in\mathbb{Z}^{n}}
\norm{D^{\beta}u}^{\tilde{R}-1}_{L^{2(\tilde{R}-1)}_{t}L^{\infty}_{x}(\R\times B_{\alpha})}\\
\lesssim
&\sum_{\tilde{R}=m+1}^{M+1}\sum_{|\gamma|\leq7+[n/2]}\lambda_{1}(D^{\gamma}u)\sum_{|\beta|\leq3}\sum_{\alpha\in\mathbb{Z}^{n}}
\norm{D^{\beta}u}^{\tilde{R}-1}_{(L^{2m}_{t}L^{\infty}_{x})\cap L^{\infty}_{t,x}(\R\times B                               _{\alpha})}\\
\lesssim
&\sum_{\tilde{R}=m+1}^{M+1}\sum_{|\gamma|\leq7+[n/2]}\lambda_{1}(D^{\gamma}u)
\sum_{|\beta|\leq3}(\lambda_{2}(D^{\beta}u)^{\tilde{R}-1}+\lambda_{3}(D^{\beta}u)^{\tilde{R}-1})\\
\lesssim &\sum_{\tilde{R}=m+1}^{M+1}\varrho^{\tilde{R}}.
\end{align*}
Combining \eqref{sct1} with \eqref{nsct2}, we get
\begin{equation*}
\sum_{|\beta|\leq3}\lambda_{2}(D^{\beta}\mathscr{T}u)\lesssim\norm{u_0}_{H^{n/2+3+\rho}}+\sum_{\tilde{R}=m+1}^{M+1}\varrho^{\tilde{R}}.
\end{equation*}

\emph{Step 3.} Now, we estimate $\lambda_{3}(D^{\beta}\mathscr{T}u)$
for $|\beta|\leq3$. In view of Proposition \ref{p29},  one has that
\begin{align*}
&\lambda_{3}(D^{\beta}\mathscr{T}u)\\
\lesssim &
\norm{S_{\eps}(t)D^{\beta}u_0}_{\ell^{2^*}_{\alpha}(L^{2m}_{t}L^{\infty}_{x}(\R\times
Q_{\alpha}))}+\norm{\mathscr{A}_{\eps}D^{\beta}F((\partial_x^\alpha
u)_{\abs{\alpha}\leq3})}_{\ell^{2^*}_{\alpha}(L^{2m}_{t}L^{\infty}_{x}(\R\times
Q_{\alpha}))}\\
\lesssim &
\norm{D^{\beta}u_0}_{H^{n/2-2/m+\rho}}+\sum_{|\beta|\leq3}\norm{D^{\beta}F((\partial_x^\alpha u)_{\abs{\alpha}\leq3})}_{L^{1}(\R,H^{n/2-2/m+\rho}(\R^n))}\\
\lesssim &
\norm{u_0}_{H^{n/2+3+\rho}}+\sum_{|\beta|\leq4+[\frac{n}{2}]}\sum_{\alpha\in\mathbb{Z}^n}\norm{D^{\beta}F((\partial_x^\alpha
u)_{\abs{\alpha}\leq3})}_{L^{1}_{t}L^{2}_{x}(\R\times Q_{\alpha}))},
\end{align*}
which reduces to the case as in \eqref{sct1}.

Therefore, collecting the estimates as in Step 1-3, we have for
$n\geq5$
\begin{equation*}
\sum_{|\beta|\leq\
[n/2]+7}\lambda_{1}(D^{\beta}\mathscr{T}u)+\sum_{|\beta|\leq3}\sum_{i=2,3}\lambda_{i}(D^{\beta}\mathscr{T}u)
\lesssim\norm{u_0}_{H^s}+\sum_{\tilde{R}=m+1}^{M+1}\varrho^{\tilde{R}},
\end{equation*}
and for $n=3,4$
\begin{equation*}
\sum_{|\beta|\leq\
[n/2]+7}\lambda_{1}(D^{\beta}\mathscr{T}u)+\sum_{|\beta|\leq3}\sum_{i=2,3}\lambda_{i}(D^{\beta}\mathscr{T}u)
\lesssim\norm{u_0}_{H^{s}\cap\dot{H}^{-3/2}}+\sum_{\tilde{R}=m+1}^{M+1}\varrho^{\tilde{R}}.
\end{equation*}
It follows that for $n\geq5$, $T:\mathscr{D}_{n}\rightarrow
\mathscr{D}_{n}$ is a contraction mapping if both $\varrho$ and
$\norm{u_0}_{H^s}$ are small
enough (similarly for $n=3,4$). \\

Before considering the case $s>n/2+9/2$, we first recall a nonlinear
mapping estimate (cf.\cite{Wang}).
\begin{lemma}\label{l42}
Let $n\geq2$, $s>0$, $K\in\mathbb{N}$. Let $1\leq p,p_{i}, q,
q_{i}\leq\infty$ satisfy $1/p=1/p_{1}+(K-1)/p_{2}$ and
$1/q=1/q_{1}+(K-1)/q_{2}$. We have
\begin{align}
\norm{v_{1}\cdots
v_{K}}_{\ell^{1,s}_{\triangle}\ell^{1}_{\alpha}(L^{q}_{t}L^{p}_{x}(\R\times
Q_{\alpha}))}\lesssim
&\sum_{k=1}^{K}\norm{v_{k}}_{\ell^{1,s}_{\triangle}\ell^{\infty}_{\alpha}(L^{q_{1}}_{t}L^{p_{1}}_{x}(\R\times
Q_{\alpha}))}\nonumber\\
&\times\prod_{i\neq
k,i=1,...,K}\norm{v_{i}}_{\ell^{1}_{\triangle}\ell^{K-1}_{\alpha}(L^{q_{2}}_{t}L^{p_{2}}_{x}(\R\times
Q_{\alpha}))}.\label{nm}
\end{align}
\end{lemma}
\begin{lemma}\label{l43}
Let $n\geq5$,  for any $s>0$ and any multi-index $\alpha$, we have
\begin{align*}
\sum_{\abs{\alpha}=0,3}\norm{S_{\eps}(t)D^{\alpha}u_0}_{\ell^{1,s}_{\triangle}\ell^{\infty}_{\alpha}(L^{2}_{t,x}(\R\times
Q_{\alpha}))}\lesssim &\norm{u_0}_{B^{s+3/2}_{2,1}},\\
\sum_{\abs{\alpha}=0,3}\norm{\mathscr{A}_{\eps}D^{\alpha}F}_{\ell^{1,s}_{\triangle}\ell^{\infty}_{\alpha}(L^{2}_{t,x}(\R\times
Q_{\alpha}))}\lesssim
&\norm{F}_{\ell^{1,s}_{\alpha}\ell^{1}_{\alpha}(L^{2}_{t,x}(\R\times
Q_{\alpha}))}.
\end{align*}
\begin{proof}
In view of Corollary \ref{c35} and Propositions \ref{p31} and
\ref{p36},  we have the desired results.
\end{proof}
\end{lemma}
\begin{lemma}\label{l44}
let $n=3,4$, for any $s>0$ and any multi-index $\alpha$, we have
\begin{align*}
\sum_{\abs{\alpha}=0,3}\norm{S_{\eps}(t)D^{\alpha}u_0}_{\ell^{1,s}_{\triangle}\ell^{\infty}_{\alpha}(L^{2}_{t,x}(\R\times
Q_{\alpha}))}\lesssim
&\norm{u_0}_{B^{s+3/2}_{2,1}\cap\dot{H}^{-3/2}},\\
\norm{\mathscr{A}_{\eps}D^{\alpha}F}_{\ell^{1,s}_{\triangle}\ell^{\infty}_{\alpha}(L^{2}_{t,x}(\R\times
Q_{\alpha}))}\lesssim
&\norm{F}_{\ell^{1,s}_{\alpha}\ell^{1}_{\alpha}(L^{2}_{t,x}(\R\times
Q_{\alpha}))}, \quad \abs{\alpha}=3,\\
\norm{\mathscr{A}_{\eps}F}_{\ell^{1,s}_{\triangle}\ell^{\infty}_{\alpha}(L^{2}_{t,x}(\R\times
Q_{\alpha}))}\lesssim
&\norm{F}_{\ell^{1,s}_{\alpha}\ell^{1}_{\alpha}(L^{1}_{t}L^{2}_{x}(\R\times
Q_{\alpha}))}.
\end{align*}

\begin{proof}By Propositions \ref{p32},  \ref{p33} and \ref{p36},  we have the
results, as desired.
\end{proof}
\end{lemma}
 We now continue the proof of Theorem \ref{t11} and consider
 the general case $s>n/2+9/2$. We write
\begin{align}
\lambda_{1}(v):=&\sum_{i=0,3}\norm{D^{i}v}_{\ell^{1,s-3/2}_{\triangle}\ell^{\infty}_{\alpha}(L^{2}_{t,x}(\R\times
Q_{\alpha}))}\nonumber,\\
\lambda_{2}(v):=&\sum_{i=0,3}\norm{D^{i}v}_{\ell^{1}_{\triangle}\ell^{2^*}_{\alpha}(L^{\infty}_{t,x}(\R\times
Q_{\alpha}))}\nonumber,\\
\lambda_{3}(v):=&\sum_{i=0,3}\norm{D^{i}v}_{\ell^{1}_{\triangle}\ell^{2^*}_{\alpha}(L^{2m}_{t}L^{\infty}_{x}(\R\times
Q_{\alpha}))}\nonumber,\\
\mathscr{D}=&\{v:\sum_{i=1,2,3}\lambda_{i}(v)\leq\varrho\}.\label{ddd}
\end{align}
Note $\lambda_{i}$ and $\mathscr{D}$ defined here are different from
those in the above. We only give the details of the proof for the
case $n\geq5$. The cases $n=3,4$ can be showed by a slight
modification. Let $\mathscr{T}$ be defined  as in \eqref{ct}.  Using
Lemma \ref{l43}, we have
\begin{equation}
\lambda_{1}(\mathscr{T}u)\lesssim\norm{u_0}_{B^{s}_{2,1}}+\norm{F}_{\ell^{1,s-3/2}_{\triangle}\ell^{1}_{\alpha}(L^{2}_{t,x}(\R\times
Q_{\alpha}))}.\label{414}
\end{equation}
For simplicity, we write
\begin{equation}
F=u^{R_{0}}(D^{\alpha_{1}}
u)^{R_{1}}(D^{\alpha_{2}}u)^{R_{2}}(D^{\alpha_{3}}u)^{R_{3}},\label{SF}
\end{equation}
and $\tilde{R}=R_{0}+|R_{1}|+|R_{2}|+|R_{3}|$. Here
$\abs{\alpha_{i}}=i$ $(i=1,2,3)$ are multi-index. By Lemma
\ref{l42}, we have
\begin{align}
&\norm{F}_{\ell^{1,s-3/2}_{\triangle}\ell^{1}_{\alpha}(L^{2}_{t,x}(\R\times
Q_{\alpha}))}\nonumber\\
&\lesssim\left(\sum_{\abs{\alpha}=0}^{3}\norm{D^{\alpha}u}_{\ell^{1,s-3/2}_{\triangle}\ell^{1}_{\alpha}(L^{2}_{t,x}(\R\times
Q_{\alpha}))}\right)\times\left(\sum_{\abs{\beta}=0}^{3}\norm{D^{\beta}u}^{\tilde{R}-1}_{\ell^{1}_{\triangle}\ell^{\tilde{R}-1}_{\alpha}(L^{\infty}_{t,x}(\R\times
Q_{\alpha}))}\right)\label{FI}
\end{align}
Using Lemma \ref{l41}, the terms of $\abs{\alpha}=1,2$ and
$\abs{\beta}=1,2$ in the above can be dominated by that of
$\abs{\alpha}=0,3$ and $\abs{\beta}=0,3$, respectively. Therefore,
we get
\begin{align*}
&\norm{F}_{\ell^{1,s-3/2}_{\triangle}\ell^{1}_{\alpha}(L^{2}_{t,x}(\R\times
Q_{\alpha}))}\nonumber\\
&\lesssim\left(\sum_{\abs{\alpha}=0,3}\norm{D^{\alpha}u}_{\ell^{1,s-3/2}_{\triangle}\ell^{\infty}_{\alpha}(L^{2}_{t,x}(\R\times
B_{\alpha}))}\right)\times\left(\sum_{\abs{\beta}=0,3}\norm{D^{\beta}u}^{\tilde{R}-1}_{\ell^{1}_{\triangle}
\ell^{\tilde{R}-1}_{\alpha}(L^{\infty}_{t,x}(\R\times
B_{\alpha}))}\right)
\end{align*}
Hence, if $u\in \mathscr{D}$, in view of \eqref{414} and \eqref{FI},
we have
\begin{equation}
\lambda_{1}(\mathscr{T}u)\lesssim\norm{u_0}_{B^{s}_{2,1}}+\sum_{m+1\leq|\tilde{R}|\leq
M+1}\varrho^{\tilde{R}}.\label{gct1}
\end{equation}
In view of the estimate for the maximal function as in Proposition
\ref{p25}, one has that
\begin{equation}
\lambda_{2}(S_{\eps}(t)u_0)\lesssim\norm{u_0}_{B^{s}_{2,1}},\label{gct21}
\end{equation}
and for multi-index $\abs{\alpha}=0,3$
\begin{align}
\norm{\mathscr{A}_{\eps}D^{\alpha}F}_{\ell^{1}_{\triangle}\ell^{2^*}_{\alpha}(L^{\infty}_{t,x}(\R\times
Q_{\alpha}))}\leq
&\sum_{j=0}^{\infty}\int_{\R}\norm{S(t-\tau)(\triangle_{j}D^{\alpha}F)(\tau)}_{\ell^{2^*}_{\alpha}(L^{\infty}_{t,x}(\R\times
Q_{\alpha}))}\nonumber\\
\lesssim
&\sum_{j=0}^{\infty}\int_{\R}\norm{(\triangle_{j}D^{\alpha}F)(\tau)}_{H^{s-9/2}(\R^{n})}d\tau\nonumber\\
\lesssim
&\sum_{j=0}^{\infty}2^{(s-3/2)j}\int_{\R}\norm{(\triangle_{j}F)(\tau)}_{L^{2}(\R^{n})}d\tau.\label{gsct2}
\end{align}
Hence, by \eqref{gct1} and \eqref{gct21}, it follows
\begin{equation}
\lambda_{2}(\mathscr{T}u)\lesssim\norm{u_0}_{B^{s}_{2,1}}+\norm{F}_{\ell^{1,s-3/2}_{\triangle}\ell^{1}_{\alpha}(L^{1}_{t}L^{2}_{x}(\R\times
Q_{\alpha}))}.\label{gct2}
\end{equation}
Similar to \eqref{gsct2}, in view of Proposition \ref{p29},  we have
\begin{equation}
\lambda_{3}(\mathscr{T}u)\lesssim\norm{u_0}_{B^{s}_{2,1}}+\norm{F}_{\ell^{1,s-3/2}_{\triangle}\ell^{1}_{\alpha}(L^{1}_{t}L^{2}_{x}(\R\times
Q_{\alpha}))}.\label{gct3}
\end{equation}
In view of Lemma \ref{l41} and \ref{l42},  we have
\begin{align*}
&\norm{F}_{\ell^{1,s-3/2}_{\triangle}\ell^{1}_{\alpha}(L^{1}_{t}L^{2}_{x}(\R\times
Q_{\alpha}))}\\
\lesssim
&\left(\sum_{\abs{\beta}=0,3}\norm{D^{\beta}u}^{\tilde{R}-1}_{\ell^{1}_{\triangle}\ell^{(\tilde{R}-1)}_{\alpha}
(L^{2(\tilde{R}-1)}_{t}L^{\infty}_{x}(\R\times
B_{\alpha}))}\right)\times\left(\sum_{\abs{\alpha}=0,3}\norm{D^{\alpha}u}_{\ell^{1,s-3/2}_{\triangle}\ell^{\infty}_{\alpha}(L^{2}_{t,x}(\R\times
B_{\alpha}))}\right).
\end{align*}
Hence,  if $u\in \mathscr{D}$,  we have
\begin{equation*}
\lambda_{2}(\mathscr{T}u)+\lambda_{3}(\mathscr{T}u)\lesssim\norm{u_0}_{B^{s}_{2,1}}+\sum_{m+1\leq|\tilde{R}|\leq
M+1}\varrho^{\tilde{R}}.
\end{equation*}
Repeating the procedures as in the above,  we obtain that there
exists a unique $u\in \mathscr{D}$ satisfying the integral equation
$\mathscr{T}u=u$, which finishes the proof of Theorem \ref{t11}.
\end{proof}

\section{Proof of Theorem 1.2}

We prove Theorem \ref{t14}  by following some idea as in Wang and
Wang \cite{Wang}. The following is the estimate for the solutions
of the fourth oredr Schr\"{o}dinger equation, see Kenig, Ponce and
Vega \cite{KPV1} and Hao, Hsiao and Wang \cite{Hao}. Recall that
$\triangle_{j}:=\mathscr{F}^{-1}\delta(2^{-j}\cdot)\mathscr{F}$,
$j\in\mathbb{Z}$ and $\delta(\cdot)$ are as in Section
\ref{functionspace}.

\begin{lemma}\label{l71}
Let $g\in\mathscr{S}(\R),f\in\mathscr{S}(\R^{2})$, we have
\begin{align}
\norm{\triangle_{j}S_{0}(t)g}_{L^{\infty}_{t}L^{2}_{x}\cap
L^{10}_{t,x}}\lesssim &\norm{\triangle_{j}g}_{2},\label{1str1}\\
\norm{\triangle_{j}S_{0}(t)g}_{L^{p}_{x}L^{\infty}_{t}}\lesssim &
2^{j(\frac{1}{2}-\frac{1}{p})}\norm{\triangle_{j}g}_{2},\label{1str2}\\
\norm{\triangle_{j}S_{0}(t)g}_{L^{\infty}_{x}L^{2}_{t}}\lesssim &
2^{-\frac{3j}{2}}\norm{\triangle_{j}g}_{2},\label{1str3}\\
\norm{\triangle_{j}\mathscr{A}_{0}f}_{L^{\infty}_{t}L^{2}_{x}\cap
L^{10}_{t,x}}\lesssim &\norm{\triangle_{j}f}_{L^{10/9}_{x,t}},\label{1str4}\\
\norm{\triangle_{j}\mathscr{A}_{0}f}_{L^{p}_{x}L^{\infty}_{t}}\lesssim
&
2^{j(\frac{1}{2}-\frac{1}{p})}\norm{\triangle_{j}f}_{L^{10/9}_{x,t}},\label{1str5}\\
\norm{\triangle_{j}\mathscr{A}_{0}f}_{L^{\infty}_{x}L^{2}_{t}}\lesssim
&
2^{-\frac{3j}{2}}\norm{\triangle_{j}f}_{L^{10/9}_{x,t}},\label{1str6}
\end{align}
and
\begin{align}
\norm{\triangle_{j}\mathscr{A}_{0}(\partial^{3}_{x}f)}_{L^{\infty}_{t}L^{2}_{x}\cap
L^{10}_{t,x}}\lesssim &2^{\frac{3j}{2}}\norm{\triangle_{j}f}_{L^{1}_{x}L^{2}_{t}},\label{1str7}\\
\norm{\triangle_{j}\mathscr{A}_{0}(\partial^{3}_{x}f)}_{L^{p}_{x}L^{\infty}_{t}}\lesssim
&
2^{\frac{3j}{2}}2^{j(\frac{1}{2}-\frac{1}{p})}\norm{\triangle_{j}f}_{L^{1}_{x}L^{2}_{t}},\label{1str8}\\
\norm{\triangle_{j}\mathscr{A}_{0}(\partial^{3}_{x}f)}_{L^{\infty}_{x}L^{2}_{t}}\lesssim
&\norm{\triangle_{j}f}_{L^{1}_{x}L^{2}_{t}}.\label{1str9}
\end{align}
\end{lemma}
For convenience, we write for any Banach function space X,
$$\norm{f}_{\ell^{1,s}_{\triangle}(X)}=\sum_{j\in\mathbb{Z}}2^{sj}\norm{\triangle_{j}f}_{X},\ \ \
  \ \norm{f}_{\ell^{1}_{\triangle}(X)}:= \norm{f}_{\ell^{1,0}_{\triangle}(X)}.    $$
Now we recall a result of Wang and Wang, see \cite{Wang}.
\begin{lemma}\label{l72}
Let $s>0$, $1\leq p,p_{i},\gamma,\gamma_{i}\leq\infty$ satisfy
\begin{equation*}
\frac{1}{p}=\frac{1}{p_{1}}+...+\frac{1}{p_{N}}, \quad
\frac{1}{\gamma}=\frac{1}{\gamma_{1}}+...+\frac{1}{\gamma_{N}},
\end{equation*}
then
\begin{align}
\norm{u_{1}\cdots
u_{N}}_{\ell^{1,s}_{\triangle}(L^{p}_{x}L^{\gamma}_{t})}\lesssim &
\norm{u_{1}}_{\ell^{1,s}_{\triangle}(L^{p_{1}}_{x}L^{\gamma_{1}}_{t})}\prod_{i=2}^{N}
\norm{u_{i}}_{\ell^{1}_{\triangle}(L^{p_{i}}_{x}L^{\gamma_{i}}_{t})}\nonumber\\
&+\norm{u_{2}}_{\ell^{1,s}_{\triangle}(L^{p_{2}}_{x}L^{\gamma_{2}}_{t})}\prod_{i\neq2,i=1,...,N}
\norm{u_{i}}_{\ell^{1}_{\triangle}(L^{p_{i}}_{x}L^{\gamma_{i}}_{t})}\nonumber\\
&+\cdots+\norm{u_N}_{\ell^{1,s}_{\triangle}(L^{p_{N}}_{x}L^{\gamma_{N}}_{t})}\prod_{i=1}^{N-1}\norm{u_{i}}_{\ell^{1}_{\triangle}(L^{p_{i}}_{x}L^{\gamma_{i}}_{t})}.\label{nm2}
\end{align}

In particular, if $u_{1}=\cdots=u_{N}=u$, then
\begin{equation}
\norm{u^{N}}_{\ell^{1,s}_{\triangle}(L^{p}_{x}L^{\gamma}_{t})}
\lesssim\norm{u}_{\ell^{1,s}_{\triangle}(L^{p_{1}}_{x}L^{\gamma_{1}}_{t})}\prod_{i=2}^{N}
\norm{u}_{\ell^{1}_{\triangle}(L^{p_{i}}_{x}L^{\gamma_{i}}_{t})}.\label{nms}
\end{equation}
Replace the spaces $L^{p}_{x}L^{\gamma}_{t}$ and
$L^{p_{i}}_{x}L^{\gamma_{i}}_{t}$ by $L^{\gamma}_{t}L^{p}_{x}$ and
$L^{\gamma_{i}}_{t}L^{p_{i}}_{x}$, respectively, \eqref{nm2} and
\eqref{nms} also hold.
\end{lemma}

\begin{remark}[\cite{Wang}]\label{rmk7.3}
One easily sees that \eqref{nms} can be sightly improved by
\begin{equation}
\norm{u^{N}}_{\ell^{1,s}_{\triangle}(L^{p}_{x}L^{\gamma}_{t})}\lesssim\norm{u}_{\ell^{1,s}_{\triangle}(L^{p_{1}}_{x}L^{\gamma_{1}}_{t})})\prod_{i=2}^{N}
\norm{u}_{L^{p_{i}}_{x}L^{\gamma_{i}}_{t}}.\label{nmt}
\end{equation}
In fact, from Minkowski's inequality it follows that
\begin{equation}
\norm{S_{r}u}_{L^{p}_{x}L^{\gamma}_{t}}\lesssim\norm{u}_{L^{p}_{x}L^{\gamma}_{t}}.\label{bnm}
\end{equation}
From Lemma \ref{l72} and \eqref{bnm}, we get \eqref{nmt}.
\end{remark}

\begin{proof}[Proof of Theorem \ref{t14}]
We can assume, without loss of generality, that
\begin{equation*}
F((\partial_x^\alpha u)_{\abs{\alpha}\leq 3},
(\partial_x^\alpha\bar{u})_{\abs{\alpha}\leq
3})=\sum_{m+1\leq\tilde{R}\leq
M+1}\lambda_{R_{0}R_{1}R_{2}R_{3}}u^{R_0}u^{R_{1}}_{x}u^{R_2}_{xx}u^{R_3}_{xxx}.
\end{equation*}
where $\tilde{R}$ is as before.

\emph{Step 1.} We consider the case $m>8$. Recall that
\begin{align}
\norm{u}_{X}=&\sup_{s_{m}\leq s\leq
\tilde{s}_{M}}\sum_{i=0,3}\sum_{j\in\mathbb{Z}}|\!|\!|\partial^i_x\triangle_{j}
u|\!|\!|_{s},\label{1sum}\\
|\!|\!|\triangle_{j}v|\!|\!|_{s}:=&2^{sj}(\norm{\triangle_{j}v}_{L^{\infty}_{t}L^{2}_{x}\cap
L^{10}_{t,x}}+2^{\frac{3j}{2}}\norm{\triangle_{j}v}_{L^{\infty}_{x}L^{2}_{t}})\nonumber\\
&+2^{(s-\tilde{s}_{m})j}\norm{\triangle_{j}v}_{L^{m}_{x}L^{\infty}_{t}}+2^{(s-\tilde{s}_{M})j}
\norm{\triangle_{j}v}_{L^{M}_{x}L^{\infty}_{t}}.\nonumber
\end{align}
Considering the mapping
\begin{equation*}
\mathscr{T}:u(t)\mapsto
S_{0}(t)u_{0}-i\mathscr{A}_{0}F((\partial_x^\alpha
u)_{\abs{\alpha}\leq 3},
(\partial_x^\alpha\bar{u})_{\abs{\alpha}\leq 3}),
\end{equation*}
we will show that $\mathscr{T}:X\rightarrow X$ is a contraction
mapping. We have
\begin{equation}
\norm{\mathscr{T}u}_{X}\lesssim\norm{S_{0}(t)u_0}_{X}+\norm{\mathscr{A}_{0}F((\partial_x^\alpha
u)_{\abs{\alpha}\leq 3},
(\partial_x^\alpha\bar{u})_{\abs{\alpha}\leq 3})}_{X}.\label{1ct}
\end{equation}
In view of \eqref{1str1}, \eqref{1str2} and \eqref{1str3}, we have
\begin{equation*}
|\!|\!|\partial^{i}_{x}\triangle_{j}S_{0}(t)u_{0}|\!|\!|_{s}\lesssim2^{sj}\norm{\partial^{i}_{x}\triangle_{j}u_0}_{2}.
\end{equation*}
It follows  that
\begin{equation}
\norm{S_{0}(t)u_0}_{X}\lesssim\sup_{s_{m}\leq s\leq
\tilde{s}_{M}}\sum_{i=0,3}\sum_{j\in\mathbb{Z}}2^{sj}
\norm{\partial^{i}_{x}\triangle_{j}u_0}_{2}\lesssim\norm{u_0}_{\dot{B}^{3+\tilde{s}_{M}}_{2,1}\cap\dot{B}^{s_{m}}_{2,1}}.\label{1q}
\end{equation}
We now estimate $\norm{\mathscr{A}_{0}F((\partial_x^\alpha
u)_{\abs{\alpha}\leq 3},
(\partial_x^\alpha\bar{u})_{\abs{\alpha}\leq 3})}_{X}$. Using
\eqref{1str4}, \eqref{1str5} and \eqref{1str6}, we have
\begin{align}
&|\!|\!|\triangle_{j}(\mathscr{A}_{0}F((\partial_x^\alpha
u)_{\abs{\alpha}\leq 3},
(\partial_x^\alpha\bar{u})_{\abs{\alpha}\leq
3}))|\!|\!|_{s}\lesssim2^{sj}\norm{\triangle_{j}F((\partial_x^\alpha
u)_{\abs{\alpha}\leq 3},
(\partial_x^\alpha\bar{u})_{\abs{\alpha}\leq
3})}_{L^{10/9}_{t,x}}.\label{1bs1}
\end{align}
From \eqref{1str7}, \eqref{1str8} and \eqref{1str9}, it follows that
\begin{align}
&|\!|\!|\triangle_{j}(\mathscr{A}_{0}\partial^{3}_{x}F((\partial_x^\alpha
u)_{\abs{\alpha}\leq 3},
(\partial_x^\alpha\bar{u})_{\abs{\alpha}\leq
3}))|\!|\!|_{s}\lesssim2^{sj}2^{3j/2}\norm{\triangle_{j}F((\partial_x^\alpha
u)_{\abs{\alpha}\leq 3},
(\partial_x^\alpha\bar{u})_{\abs{\alpha}\leq
3})}_{L^{1}_{x}L^{2}_{t}}.\label{1bs2}
\end{align}
Hence, from \eqref{1sum}, \eqref{1bs1} and \eqref{1bs2} we have
\begin{align}
&\norm{\mathscr{A}_{0}F((\partial_x^\alpha u)_{\abs{\alpha}\leq 3}, (\partial_x^\alpha\bar{u})_{\abs{\alpha}\leq 3})}_{X}\nonumber\\
\lesssim&2^{sj}\norm{\triangle_{j}F((\partial_x^\alpha u)_{\abs{\alpha}\leq 3}, (\partial_x^\alpha\bar{u})_{\abs{\alpha}\leq 3})}_{L^{10/9}_{t,x}}\nonumber\\
&+2^{sj}2^{3j/2}\norm{\triangle_{j}F((\partial_x^\alpha
u)_{\abs{\alpha}\leq 3},
(\partial_x^\alpha\bar{u})_{\abs{\alpha}\leq
3})}_{L^{1}_{x}L^{2}_{t}}=I+II.\label{FJ}
\end{align}
Now we perform the nonlinear estimates. By Lemma \ref{l72}, we have
\begin{align*}
I\lesssim &\sum_{m+1\leq\tilde{R}\leq
M+1}\bigg(\norm{u}_{\ell^{1,s}_{\triangle}(L^{10}_{t,x})}\norm{u}^{R_{0}-1}_{\ell^{1}_{\triangle}
(L^{(5(\tilde{R}-1)/4}_{t,x})}\prod_{i=1}^{3}\norm{\partial^{i}_{x}}^{R_{i}}_{\ell^{1}_{\triangle}
(L^{(5(\tilde{R}-1)/4}_{t,x})}\\
&+\norm{\partial_{x}u}_{\ell^{1,s}_{\triangle}(L^{10}_{t,x})}
\norm{\partial_{x}u}^{R_{1}-1}_{\ell^{1}_{\triangle}(L^{(5(\tilde{R}-1)/4}_{t,x})}
\prod_{i=0,2,3}\norm{\partial^{i}_{x}u}^{R_{i}}_{\ell^{1}_{\triangle}(L^{(5(\tilde{R}-1)/4}_{t,x})}\\
&+\norm{\partial^{2}_{x}u}_{\ell^{1,s}_{\triangle}(L^{10}_{t,x})}
\norm{\partial^{2}_{x}u}^{R_{2}-1}_{\ell^{1}_{\triangle}(L^{(5(\tilde{R}-1)/4}_{t,x})}
\prod_{i=0,1,3}\norm{\partial^{i}_{x}u}^{R_{i}}_{\ell^{1}_{\triangle}(L^{(5(\tilde{R}-1)/4}_{t,x})}\\
&+\norm{\partial^{3}_{x}u}_{\ell^{1,s}_{\triangle}(L^{10}_{t,x})}
\norm{\partial^{3}_{x}u}^{R_{3}-1}_{\ell^{1}_{\triangle}(L^{(5(\tilde{R}-1)/4}_{t,x})}
\prod_{i=0}^{2}\norm{\partial^{i}_{x}u}^{R_{i}}_{\ell^{1}_{\triangle}(L^{(5(\tilde{R}-1)/4}_{t,x})}\bigg)\\
\lesssim&\sum_{m+1\leq\tilde{R}\leq
M+1}\left(\sum_{i=0}^{3}\norm{\partial^{i}_{x}u}_{\ell^{1,s}_{\triangle}(L^{10}_{t,x})}\right)
\left(\sum_{i=0}^{3}\norm{\partial^{i}_{x}u}^{\tilde{R}-1}_{\ell^{1}_{\triangle}(L^{5(\tilde{R}-1)/4}_{t,x})}\right).
\end{align*}
By Sobolev imbedding theorem, we have
\begin{equation}
I\lesssim\sum_{m+1\leq\tilde{R}\leq
M+1}\left(\sum_{i=0,3}\norm{\partial^{i}_{x}u}_{\ell^{1,s}_{\triangle}(L^{10}_{t,x})}\right)\left(\sum_{i=0,3}
\norm{\partial^{i}_{x}u}^{\tilde{R}-1}_{\ell^{1}_{\triangle}(L^{5(\tilde{R}-1)/4}_{t,x})}\right).\label{1I}
\end{equation}
For any $m\leq \lambda\leq M$,  let
$\rho=\frac{1}{2}-\frac{16}{5\lambda}$. Observing that the following
inclusions hold:
\begin{equation}
L^{\infty}_{t}(\R,\dot{H}^{s_{\lambda}})\cap
L^{10}_{t}(\dot{H}^{s_{\lambda}}_{10})\subset
L^{5\lambda/4}_{t}(\R,\dot{H}^{s_{\lambda}})\subset
L^{5\lambda/4}_{x,t}\label{IB}.
\end{equation}
More precisely, we have
\begin{align}
\sum_{j\in\mathbb{Z}}\norm{\triangle_{j}u}_{L^{5\lambda/4}_{t,x}}\lesssim &\sum_{j\in\mathbb{Z}}\norm{\triangle_{j}u}_{L^{5\lambda/4}_{t}(\R,\dot{H}^{s_{\lambda}}_{\rho})}\nonumber\\
\lesssim &\sum_{j\in\mathbb{Z}}\norm{\triangle_{j}u}^{8/\lambda}_{L^{10}_{t}(\R,\dot{H}^{s_{\lambda}}_{10})}\norm{\triangle_{j}u}^{1-8/\lambda}_{L^{\infty}_{t}(\R,\dot{H}^{s_{\lambda}})}\nonumber\\
\lesssim
&\norm{u}^{8/\lambda}_{\ell^{1,s_{\lambda}}_{\triangle}(L^{10}_{x,t})}
\norm{u}^{1-8/\lambda}_{\ell^{1,s_{\lambda}}_{\triangle}(L^{\infty}_{t}L^{2}_{x})}.\label{1Ib}
\end{align}
Using \eqref{1Ib} and noticing that $s_{m}\leq s_{\tilde{R}-1}\leq
s_{M}<\tilde{s}_{M}$, for $i=0,3$, we have
\begin{align}
\norm{\partial^{i}_{x}u}^{\tilde{R}-1}_{\ell^{1}_{\triangle}(L^{5\lambda/4}_{t,x})}
\lesssim
&\norm{\partial^{i}_{x}u}^{8/\lambda}_{\ell^{1,s_{\lambda}}_{\triangle}(L^{10}_{x,t})}\norm{\partial^{i}_{x}u}^{1-8/\lambda}_{\ell^{1,s_{\lambda}}_{\triangle}(L^{\infty}_{t}L^{2}_{x})}
\lesssim \norm{u}^{\tilde{R}-1}_{X}.\label{1Ip}
\end{align}
Combining \eqref{1I} with \eqref{1Ip}, we have
\begin{equation}
I\lesssim\sum_{m+1\leq\tilde{R}\leq
M+1}\norm{u}^{\tilde{R}}_{X}.\label{IR}
\end{equation}
 Now we estimate the term $II$.  By
Lemma \ref{l72}, we have
\begin{align}
II\lesssim &\sum_{m+1\leq\tilde{R}\leq
M+1}\bigg(\norm{u}_{\ell^{1,s+3/2}_{\triangle}(L^{\infty}_{x}L^{2}_{t})}
\norm{u}^{R_{0}-1}_{\ell^{1}_{\triangle}(L^{\tilde{R}-1}_{x}L^{\infty}_{t})}\prod_{i=1}^{3}\norm{\partial^{i}_{x}}^{R_{i}}_{\ell^{1}_{\triangle}(L^{\tilde{R}-1}_{x}L^{\infty}_{t})}\nonumber\\
&+\norm{\partial_{x}u}_{\ell^{1,s+3/2}_{\triangle}(L^{\infty}_{x}L^{2}_{t})}
\norm{\partial_{x}u}^{R_{1}-1}_{\ell^{1}_{\triangle}(L^{\tilde{R}-1}_{x}L^{\infty}_{t})}\prod_{i=0,2,3}\norm{\partial^{i}_{x}u}^{R_{i}}_{\ell^{1}_{\triangle}(L^{\tilde{R}-1}_{x}L^{\infty}_{t})}\nonumber\\
&+\norm{\partial^{2}_{x}u}_{\ell^{1,s+3/2}_{\triangle}(L^{\infty}_{x}L^{2}_{t})}
\norm{\partial^{2}_{x}u}^{R_{2}-1}_{\ell^{1}_{\triangle}(L^{\tilde{R}-1}_{x}L^{\infty}_{t})}\prod_{i=0,1,3}\norm{\partial^{i}_{x}u}^{R_{i}}_{\ell^{1}_{\triangle}(L^{\tilde{R}-1}_{x}L^{\infty}_{t})}\nonumber\\
&+\norm{\partial^{3}_{x}u}_{\ell^{1,s+3/2}_{\triangle}(L^{\infty}_{x}L^{2}_{t})}
\norm{\partial^{3}_{x}u}^{R_{3}-1}_{\ell^{1}_{\triangle}(L^{\tilde{R}-1}_{x}L^{\infty}_{t})}\prod_{i=0}^{2}\norm{\partial^{i}_{x}u}^{R_{i}}_{\ell^{1}_{\triangle}(L^{\tilde{R}-1}_{x}L^{\infty}_{t})}\bigg)\nonumber\\
\lesssim &\sum_{m+1\leq\tilde{R}\leq
M+1}\left(\sum_{i=0}^{3}\norm{\partial^{i}_{x}u}_{\ell^{1,s+3/2}_{\triangle}(L^{\infty}_{x}L^{2}_{t})}\right)\left(\sum_{i=0}^{3}
\norm{\partial^{i}_{x}u}^{\tilde{R}-1}_{\ell^{1}_{\triangle}(L^{\tilde{R}-1}_{x}L^{\infty}_{t})}\right).\label{II1}
\end{align}
As before, we have
\begin{align}
II\lesssim &\sum_{m+1\leq\tilde{R}\leq M+1}\left(\sum_{i=0,3}
\norm{\partial^{i}_{x}u}_{\ell^{1,s+3/2}_{\triangle}(L^{\infty}_{x}L^{2}_{t})}\right)\left(\sum_{i=0,3}
\norm{\partial^{i}_{x}u}^{\tilde{R}-1}_{\ell^{1}_{\triangle}(L^{\tilde{R}-1}_{x}L^{\infty}_{t})}\right)\nonumber\\
\lesssim &\sum_{m+1\leq\tilde{R}\leq
M+1}\norm{u}_{X}\left(\sum_{i=0,3}
\norm{\partial^{i}_{x}u}^{\tilde{R}-1}_{\ell^{1}_{\triangle}(L^{m}_{x}L^{\infty}_{t}\cap L^{M}_{x}L^{\infty}_{t})}\right)\nonumber\\
\lesssim &\sum_{m+1\leq\tilde{R}\leq
M+1}\norm{u}^{\tilde{R}}_{X}.\label{IIR}
\end{align}
Collecting \eqref{FJ}, \eqref{1I}, \eqref{IR} and \eqref{IIR},  we
have
\begin{equation}
\norm{\mathscr{A}_{0}F((\partial_x^\alpha u)_{\abs{\alpha}\leq 3},
(\partial_x^\alpha\bar{u})_{\abs{\alpha}\leq
3})}_{X}\lesssim\sum_{m+1\leq\tilde{R}\leq
M+1}\norm{u}^{\tilde{R}}_{X}.\label{AR}
\end{equation}
By \eqref{1ct}, \eqref{1q} and \eqref{AR}, it follows
\begin{equation*}
\norm{\mathscr{T}u}_{X}\lesssim\norm{u_0}_{\dot{B}^{3+\tilde{s}_{M}}_{2,1}\cap\dot{B}^{s_{m}}_{2,1}}+\sum_{m+1\leq\tilde{R}\leq
M+1}\norm{u}^{\tilde{R}}_{X}.
\end{equation*}

\emph{Step 2.} We consider the case $m=8$. Recall that
$$
\norm{u}_{X}=\sum_{i=0,3}(\norm{\partial^{i}_{x}u}_{L^{\infty}_{t}L^{2}_{x}\cap
L^{10}_{t,x}}+\sup_{s_{9}\leq s\leq
\tilde{s}_{M}}\sum_{j\in\mathbb{Z}}|\!|\!|\partial^{i}_{x}u|\!|\!|_{s}).
$$
By \eqref{1str1}, \eqref{1str2} and \eqref{1str3}, we get
\begin{equation*}
\norm{S_{0}(t)u_0}_{X}\lesssim\norm{u_0}_{L^{2}}+\sup_{s_{9}\leq
s\leq
\tilde{s}_{M}}\sum_{i=0,3}\sum_{j\in\mathbb{Z}}2^{sj}\norm{\partial^{i}_{x}\triangle_{j}u_0}_{2}
\lesssim\norm{u_0}_{B^{3+\tilde{s}_{M}}_{2,1}}
\end{equation*}
We now estimate $\norm{\mathscr{A}_{0}F((\partial_x^\alpha
u)_{\abs{\alpha}\leq 3},
(\partial_x^\alpha\bar{u})_{\abs{\alpha}\leq 3})}_{X}$. By
Strichartz's and H\"{o}lder's inequality, we have
\begin{align}
&\norm{\mathscr{A}_{0}F((\partial_x^\alpha u)_{\abs{\alpha}\leq 3}, (\partial_x^\alpha\bar{u})_{\abs{\alpha}\leq 3})}_{L^{\infty}_{t}L^{2}_{x}\cap L^{10}_{x,t}}\nonumber\\
&\lesssim \sum_{9\leq\tilde{R}\leq M+1}\norm{F((\partial_x^\alpha u)_{\abs{\alpha}\leq 3}, (\partial_x^\alpha\bar{u})_{\abs{\alpha}\leq 3})}_{L^{10/9}_{t,x}}\nonumber\\
&\lesssim \sum_{9\leq\tilde{R}\leq M+1}\left(\sum_{i=0}^{3}\norm{\partial^{i}_{x}u}_{L^{10}_{x,t}}\right)\left(\sum_{i=0}^{3}\norm{\partial^{i}_{x}u}^{\tilde{R}-1}_{L^{5(\tilde{R}-1)/4}_{x,t}}\right)\nonumber\\
&\lesssim \sum_{9\leq\tilde{R}\leq M+1}\left(\sum_{i=0}^{3}\norm{\partial^{i}_{x}u}_{L^{10}_{x,t}}\right)\left(\sum_{i=0}^{3}\norm{\partial^{i}_{x}u}^{\tilde{R}-1}_{L^{10}_{x,t}\cap L^{5M/4}_{x,t}}\right)\nonumber\\
&\lesssim \sum_{9\leq\tilde{R}\leq
M+1}\left(\sum_{i=0}^{3}\norm{\partial^{i}_{x}u}_{L^{10}_{x,t}}\right)^{\tilde{R}}
+\sum_{9\leq\tilde{R}\leq M+1}\left(\sum_{i=0}^{3}\norm{\partial^{i}_{x}u}_{L^{10}_{x,t}}\right)\left(\sum_{i=0}^{3}\norm{\partial^{i}_{x}u}^{\tilde{R}-1}_{L^{5M/4}_{x,t}}\right)\nonumber\\
&\lesssim \sum_{9\leq\tilde{R}\leq
M+1}\left(\sum_{i=0,3}\norm{\partial^{i}_{x}u}_{L^{10}_{x,t}}\right)^{\tilde{R}}+\sum_{9\leq\tilde{R}\leq
M+1}\left(\sum_{i=0,3}\norm{\partial^{i}_{x}u}_{L^{10}_{x,t}}\right)\left(\sum_{i=0,3}\norm{\partial^{i}_{x}u}^{\tilde{R}-1}_{L^{5M/4}_{x,t}}\right).\label{II2}
\end{align}
Applying \eqref{1Ib} and \eqref{II2}, it implies that
\begin{equation}
\norm{\mathscr{A}_{0}F((\partial_x^\alpha u)_{\abs{\alpha}\leq 3},
(\partial_x^\alpha\bar{u})_{\abs{\alpha}\leq
3})}_{L^{\infty}_{t}L^{2}_{x}\cap
L^{10}_{x,t}}\lesssim\sum_{9\leq\tilde{R}\leq
M+1}\norm{u}^{\tilde{R}}_{X}.\label{II3}
\end{equation}
Using Bernstein's inequality and \eqref{1str7}, it follows that
\begin{align*}
&\norm{\partial^{3}_{x}\mathscr{A}_{0}F((\partial_x^\alpha u)_{\abs{\alpha}\leq 3}, (\partial_x^\alpha\bar{u})_{\abs{\alpha}\leq 3})}_{L^{\infty}_{t}L^{2}_{x}\cap L^{10}_{x,t}}\nonumber\\
\leq&\norm{P_{\leq1}(\mathscr{A}_{0}\partial^{3}_{x}F((\partial_x^\alpha u)_{\abs{\alpha}\leq 3}, (\partial_x^\alpha\bar{u})_{\abs{\alpha}\leq 3}))}_{L^{\infty}_{t}L^{2}_{x}\cap L^{10}_{x,t}}\nonumber\\
&+\norm{P_{>1}(\mathscr{A}_{0}\partial^{3}_{x}F((\partial_x^\alpha u)_{\abs{\alpha}\leq 3}, (\partial_x^\alpha\bar{u})_{\abs{\alpha}\leq 3}))}_{L^{\infty}_{t}L^{2}_{x}\cap L^{10}_{x,t}}\nonumber\\
\lesssim&\norm{\mathscr{A}_{0}F((\partial_x^\alpha u)_{\abs{\alpha}\leq 3}, (\partial_x^\alpha\bar{u})_{\abs{\alpha}\leq 3})}_{L^{\infty}_{t}L^{2}_{x}\cap L^{10}_{x,t}}\nonumber\\
&+\sum_{j\geq1}2^{3j/2}\norm{\triangle_{j}F((\partial_x^\alpha u)_{\abs{\alpha}\leq 3}, (\partial_x^\alpha\bar{u})_{\abs{\alpha}\leq 3})}_{L^{1}_{x}L^{2}_{t}}\nonumber\\
\lesssim&\norm{\mathscr{A}_{0}F((\partial_x^\alpha u)_{\abs{\alpha}\leq 3}, (\partial_x^\alpha\bar{u})_{\abs{\alpha}\leq 3})}_{L^{\infty}_{t}L^{2}_{x}\cap L^{10}_{x,t}}\nonumber\\
&+\sum_{j\in\mathbb{Z}}2^{\tilde{s}_{M}}2^{3j/2}\norm{\triangle_{j}F((\partial_x^\alpha
u)_{\abs{\alpha}\leq 3},
(\partial_x^\alpha\bar{u})_{\abs{\alpha}\leq
3})}_{L^{1}_{x}L^{2}_{t}}=III+IV.
\end{align*}
The estimates of $III$ and $IV$  have been given in \eqref{II3} and
\eqref{II1}, respectively. We have
\begin{equation*}
\norm{\partial^{3}_{x}\mathscr{A}_{0}F((\partial_x^\alpha
u)_{\abs{\alpha}\leq 3},
(\partial_x^\alpha\bar{u})_{\abs{\alpha}\leq
3})}_{L^{\infty}_{t}L^{2}_{x}\cap L^{10}_{x,t}}
\lesssim\sum_{9\leq\tilde{R}\leq M+1}\norm{u}^{\tilde{R}}_{X}.
\end{equation*}
We have from \eqref{1str1}-\eqref{1str3} and
\eqref{1str7}-\eqref{1str9} that
\begin{align}
&\sum_{j\in\mathbb{Z}}|\!|\!|\triangle_{j}(\mathscr{A}_{0}F((\partial_x^\alpha u)_{\abs{\alpha}\leq 3}, (\partial_x^\alpha\bar{u})_{\abs{\alpha}\leq 3}))|\!|\!|_{s}\nonumber\\
\lesssim&\sum_{j\in \mathbb{Z}}2^{sj}\norm{\triangle_{j}F((\partial_x^\alpha u)_{\abs{\alpha}\leq 3}, (\partial_x^\alpha\bar{u})_{\abs{\alpha}\leq 3})}_{L^{10/9}_{x,t}},\label{II4}\\
&\sum_{j\in\mathbb{Z}}|\!|\!|\triangle_{j}(\mathscr{A}_{0}\partial^{3}_{x}F((\partial_x^\alpha u)_{\abs{\alpha}\leq 3}, (\partial_x^\alpha\bar{u})_{\abs{\alpha}\leq 3}))|\!|\!|_{s}\nonumber\\
\lesssim&\sum_{j\in
\mathbb{Z}}2^{3j/2}2^{sj}\norm{\triangle_{j}F((\partial_x^\alpha
u)_{\abs{\alpha}\leq 3},
(\partial_x^\alpha\bar{u})_{\abs{\alpha}\leq
3})}_{L^{1}_{x}L^{2}_{t}},\label{II5}
\end{align}
holds for all $s>0$. The right hand side in \eqref{II5} has been
estimated by \eqref{II1}.  Thus, it suffices to consider the
estimate of the right hand side in \eqref{II4}. let us observe the
equality
\begin{align*}
&F((\partial_x^\alpha u)_{\abs{\alpha}\leq 3}, (\partial_x^\alpha\bar{u})_{\abs{\alpha}\leq 3})\nonumber\\
=&\sum_{\tilde{R}=9}\lambda_{R_{0}R_{1}R_{2}R_{3}}u^{R_0}u^{R_{1}}_{x}u^{R_2}_{xx}u^{R_3}_{xxx}+\sum_{9<\tilde{R}\leq M+1}\lambda_{R_{0}R_{1}R_{2}R_{3}}u^{R_0}u^{R_{1}}_{x}u^{R_2}_{xx}u^{R_3}_{xxx}\nonumber\\
=& V+VI.
\end{align*}
For any $s_{9}\leq s\leq \tilde{s}_{M}$, $VI$ has been handled in
\eqref{IB}-\eqref{IR}:
\begin{equation*}
\sum_{9<\tilde{R}\leq M+1}\sum_{j\in
\mathbb{Z}}2^{sj}\norm{\triangle_{j} F((\partial_x^\alpha
u)_{\abs{\alpha}\leq 3},
(\partial_x^\alpha\bar{u})_{\abs{\alpha}\leq 3})}_{L^{10/9}_{x,t}}
\lesssim\sum_{9<\tilde{R}\leq M+1}\norm{u}^{\tilde{R}}_{X}.
\end{equation*}
For the estimate of $V$, we use Remark \ref{rmk7.3} and get for any
$s_{9}\leq s\leq\tilde{s}_{M}$,
\begin{align*}
\sum_{\tilde{R}=9}\sum_{j\in
\mathbb{Z}}2^{sj}\norm{\triangle_{j}(u^{R_0}u^{R_{1}}_{x}u^{R_2}_{xx}u^{R_3}_{xxx})}\lesssim
&\left(\sum_{i=0}^{3}\norm{\partial^{i}_{x}u}^{8}_{L^{10}_{x,t}}\right)\left(\sum_{i=0}^{3}\norm{\partial^{i}_{x}u}_{\ell^{1,s}_{\triangle}(L^{10}_{x,t})}\right)
\lesssim \norm{u}^{9}_{X}.
\end{align*}
Summating the estimates above, we obtain
\begin{equation*}
\norm{\mathscr{T}u(t)}_{X}\lesssim\norm{u_0}_{B^{3+\tilde{s}_{M}}_{2,1}}+\sum_{9\leq\tilde{R}\leq
M+1}\norm{u}^{\tilde{R}}_{X}.
\end{equation*}
Hence, we have the results, as desired.
\end{proof}

\section*{Acknowledgments}
The author would like to express his great thanks to Professor
Baoxiang Wang and Doctor Chenchun Hao for their valuable suggestions
and frequent encouragement during every stage of this work. The
author was partially supported by the National Natural Science
Foundation of China (grants numbers 10571004 and 10621061), the 973
Project Foundation of China, grant number 2006CB805902.

\end{document}